\pdfoutput=1
\documentclass[letterpaper,11pt]{article}
\usepackage{amsmath,amsthm,amssymb}
\usepackage{graphicx, fullpage, xcolor}
\usepackage[colorinlistoftodos,bordercolor=orange,backgroundcolor=orange!20,linecolor=orange,textsize=scriptsize]{todonotes}
\usepackage[numbers]{natbib}
\usepackage{subcaption}

\usepackage{algorithm}
\usepackage{algorithmic}

\usepackage{amsmath,amssymb, amsthm}
\usepackage{amsfonts}
\usepackage{layout}
\usepackage{amsmath}
\usepackage{amsthm}
\usepackage{amssymb}
\usepackage{url}
\usepackage{color}
\usepackage{graphicx}
\usepackage{verbatim}

\newcommand{\eqdef}{:=}

\newcommand{\R}{\mathbb{R}}
\newcommand{\Prob}{\mathbf{Prob}}
\newcommand{\E}{\mathbf{E}}

\newcommand{\bU}{\mathbf U}

\newcommand{\bA}{\mathbf A}

\newcommand{\bM}{\mathbf M}

\newcommand{\bI}{\mathbf I}

\newcommand{\cS}{\mathcal{ S}}
\newcommand{\cP}{\mathcal{P}}

\newcommand{\cX}{\mathcal{X}}
\newcommand{\bX}{\mathbf{X}}

\DeclareMathOperator*{\argmin}{arg\,min}
\DeclareMathOperator*{\argmax}{arg\,max}

\newtheorem{example}{Example}

\theoremstyle{plain}
\newtheorem{theorem}{Theorem}
\newtheorem{proposition}[theorem]{Proposition}

\newtheorem{lemma}[theorem]{Lemma}

\newtheorem{assumption}[theorem]{Assumption}

\theoremstyle{definition}

\newtheorem{definition}[theorem]{Definition}

\makeatletter
\newcommand*{\rom}[1]{\expandafter\@slowromancap\romannumeral #1@}
\makeatother

\newcommand{\bb}{\mathbf{b}}

\newcommand{\bh}{\mathbf{h}}
\newcommand{\qq}[1]{\quad #1 \quad}

\newcommand{\be}{\mathbf{e}}

\newcommand{\bx}{\mathbf{x}}
\newcommand{\by}{\mathbf{y}}
\newcommand{\bu}{\mathbf{u}}

\newcommand{\cA}{\mathcal{A}}
\newcommand{\cC}{\mathcal{C}}
\newcommand{\cO}{\mathcal{O}}
\newcommand{\cL}{\mathcal{L}}
\newcommand{\cW}{\mathcal{W}}
\newcommand{\extR}{\bar{\R}}

\title{Global Convergence of 
Arbitrary-Block
 Gradient Methods \\ for Generalized  Polyak-\L{}ojasiewicz Functions}

\author{Dominik Csiba\thanks{University of Edinburgh} \and Peter Richt\'{a}rik\thanks{KAUST and University of Edinburgh}}

\begin{document}

\maketitle

\begin{abstract} 
In this paper we introduce two novel generalizations of the theory for gradient descent type methods in the proximal setting. First, we introduce the proportion function, which we further use to analyze all  known (and many new) block-selection rules for block coordinate descent methods under a single framework. This framework includes randomized methods with uniform, non-uniform or even adaptive sampling strategies, as well as deterministic methods with batch, greedy or cyclic selection rules. We additionally introduce a novel block selection technique called greedy minibatches, for which we provide competitive convergence guarantees. Second, the  theory of strongly-convex optimization was recently generalized to a specific class of non-convex functions satisfying the so-called Polyak-\L{}ojasiewicz condition. To mirror this generalization in the weakly convex case, we introduce the Weak Polyak-\L{}ojasiewicz condition, using which we give global convergence guarantees for a class of non-convex functions previously not considered in theory. Additionally, we establish (necessarily somewhat weaker) convergence guarantees for an even larger class of non-convex functions satisfying a certain smoothness assumption only.

By combining the two abovementioned generalizations we recover the state-of-the-art convergence guarantees for a large class of previously known methods and setups as special cases of our general framework. Moreover, our frameworks allows for the derivation of new guarantees for many new combinations of methods and setups, as well as a large class of novel non-convex objectives.
The flexibility of our approach offers a lot of potential for future research, as a new block selection procedure will have a convergence guarantee for all objectives considered in our framework, while a new objective analyzed under our approach will have a whole fleet of block selection rules with convergence guarantees readily available.
\end{abstract}

{\footnotesize
\tableofcontents
}

\section{Introduction}

During the last decade, gradient-type methods have become the methods of choice for solving optimization problems of very large sizes arising in fields such as machine learning, data science, engineering, and visual computing. 

Consider the optimization problem 
\[\min_{\bx\in \R^n} f(\bx),\]
where $f:\R^n\to\R$ is a differentiable function. Assume that this problem has a nonempty set of global minimizers $\cX^*$ (clearly, $\nabla f(x^*) = 0$ for all $x^*\in \cX^*$).  It is well known \cite{bubeck2015convex} that if $f$ is $L$-smooth and $\mu$-strongly convex, where  $L\geq \mu >0$, then the gradient descent method $\bx^{k+1} = \bx^k - \tfrac{1}{L}\nabla f(\bx^k)$ for all $k\geq 0$ satisfies 
$\xi(\bx^{k+1}) \leq \left(1-\tfrac{\mu}{L}\right) \xi(\bx^k),$
where $\bx^*\in \cX^*$ and 
\begin{equation}\label{eq:xi}
\xi(\bx) \qq{\eqdef} f(\bx) - f(\bx^*),
\end{equation} is the {\em optimality gap function}. Motivated by the rise of nonconvex  models in fields such as image and signal processing  and deep learning, there is interest in studying the performance of gradient-type methods for nonconvex functions. 

As observed by Polyak in 1963 \cite{polyak1963gradient}, and recently popularized and further studied by Karimi, Nutini and Schmidt \cite{karimi2016linear} in the context of proximal methods,  proofs of linear convergence  rely on a certain {\em consequence} of strong convexity known as the {\em Polyak-\L{}ojasiewics} (PL) inequality. Since functions satisfying the PL inequality need not be convex, linear convergence of gradient methods to the global optimum extends beyond the realm of convex functions. 

The (strong) PL inequality can be written in the form
\begin{equation}\label{eq:s09h9hf3}
\tfrac{1}{2}\|\nabla f(\bx)\|^2 \qq{\geq} \mu \cdot \xi(\bx), \qquad \bx\in \R^n.
\end{equation}

We write $f\in \cS_{PL}(\mu)$ if $f$ satisfies \eqref{eq:s09h9hf3}. The PL inequality and methods based on it have been an inspiration for many researchers in recent years \cite{de2016big, lei2017nonconvex, gao2016ojasiewicz}. It is known that in order to guarantee $\xi(\bx^k) \leq \epsilon$, it suffices to take $k=\cO((L/\mu)\log(1/\epsilon))$.

The starting point of this paper is the realization that {\em while the PL inequality serves as a generalization of strong convexity, there is no equivalent generalization of (weak) convexity.} One of the key contributions of this paper is to remedy this situation by introducing the {\em weak PL inequality}:
\begin{equation} \label{eq:98dh8gh3d} 
\|\nabla f(\bx)\| \cdot \|\bx-\bx^*\| \qq{\geq} \sqrt{\mu} \cdot \xi(\bx), \qquad \bx\in \R^n.
\end{equation}
We write $f\in \cW_{PL}(\mu)$ if $f$ satisfies \eqref{eq:98dh8gh3d}.
If  $f$ is convex, then $f\in \cW_{PL}(1)$. Indeed, by convexity and Cauchy-Schwartz inequality, we have
\[\xi(\bx) \qq{=} f(\bx) - f(\bx^*) \qq{\leq} \langle \nabla f(\bx), \bx-\bx^* \rangle \qq{\leq} \|\nabla f(\bx)\| \cdot \|\bx-\bx^*\|.\]

However, $\cW_{PL}(1)$ contains nonconvex functions as well. As an example consider the function $f(x_1, x_2) = x_1^2 x_2^2$, for which it is straightforward to show that $f \in \cW_{PL}(1)$ and it is apparently nonconvex.  If we allow $0<\mu<1$, the inequality \eqref{eq:98dh8gh3d} becomes weaker, and holds for a larger family of functions still.

 
In this paper we  prove that $\xi(\bx^k)\leq \epsilon$ if $k\geq 2L R^2/\epsilon$, where $R$ is a uniform upper bound on $\|\bx^k-\bx^0\|$. Since gradient descent is a monotonic method, such a bound exists if, for instance, the level set $\{\bx \;:\; f(\bx) \leq f(\bx^0)\}$ is bounded. This result extends standard convergence result for gradient descent for convex functions to weak PL functions. 

\subsection{Contributions}

We now  briefly summarize the main contributions of this work.

\begin{enumerate}
\item[(i)]  We consider a large family of gradient type methods. The methods include block coordinate descent with arbitrary block selection rules, such as cyclic, greedy, randomized, adaptive and so on. Gradient descent arises as a special case when the active  block at each iteration consists of all coordinates. Also, we introduce a novel method called greedy minibatch descent which we analyze using our developed theory to prove convergence for it in various setups mentioned in the next point.
\item[(ii)] We extend all results (strong and weak PL inequality, algorithms and  complexity results)  to the proximal setup. That is, we consider {\em composite} optimization problems  of the  form \[\min_{x\in \R^n} F(x) \eqdef f(x)+g(x),\] where $f$ is a differentiable function, and $g$ is a simple (and possibly nonsmooth) function. For instance, the weak PL inequality \eqref{eq:98dh8gh3d}  arises as a special case of the new {\em proximal weak PL inequality} when $g=0$. The complexity results are the same: $\cO(\log(1/\epsilon))$ for strongly PL functions (see Section~\ref{sec:SPL}), $\cO(1/\epsilon)$ for weakly PL functions (see Section~\ref{sec:WPLg}), and $\cO(1/\epsilon \log(1/\epsilon))$ for general nonconvex functions (see Section~\ref{sec:general_non-convex_applications}). The specific rates can be found in Tables~\ref{tab:rates_smooth} and \ref{tab:rates_nonsmooth}. The definitions of the various symbols and constants appearing in the tables is given in the rest of the text. 

\begin{table}[ht]
\centering
\begin{tabular}{|l|c|c|c|}
\hline
{\bf block selection rule} & {\bf strongly PL} &{\bf weakly PL} & {\bf general nonconvex}  \\ \hline \hline
gradient descent & $\tfrac{\lambda_{\max}(\bM)}{\mu}\log(\tfrac{\xi(\bx^0)}{\epsilon})$ & $\tfrac{\lambda_{\max}(\bM)}{\rho(\bx^0) \epsilon}$ &  $\frac{\lambda_{\max}(\bM) \xi(\bx^0)}{\epsilon} \log(\frac{\xi(\bx^0)}{\epsilon})$ \\ \hline
uniform coordinate & $\tfrac{n \max_i \{M_{ii}\}}{\mu}\log(\tfrac{\xi(\bx^0)}{\epsilon})$ & $\tfrac{n \max_i \{M_{ii}\}}{\rho(\bx^0) \epsilon}$ &  $\frac{n \max_i \{M_{ii}\} \xi(\bx^0)}{\epsilon} \log(\frac{\xi(\bx^0)}{\epsilon})$ \\ \hline
importance coordinate & $\tfrac{\sum_{i=1}^n M_{ii}}{\mu}\log(\tfrac{\xi(\bx^0)}{\epsilon})$ & $\tfrac{\sum_{i=1}^n M_{ii}}{\rho(\bx^0) \epsilon}$ &  $\frac{\sum_{i=1}^n M_{ii} \xi(\bx^0)}{\epsilon} \log(\frac{\xi(\bx^0)}{\epsilon})$ \\ \hline
greedy coordinate & $\tfrac{\sum_{i=1}^n M_{ii}}{\mu}\log(\tfrac{\xi(\bx^0)}{\epsilon})$ & $\tfrac{\sum_{i=1}^n M_{ii}}{\rho(\bx^0) \epsilon}$ &  $\frac{\sum_{i=1}^n M_{ii} \xi(\bx^0)}{\epsilon} \log(\frac{\xi(\bx^0)}{\epsilon})$ \\ \hline
uniform minibatch & $\tfrac{1}{\mu \lambda_{\min}(\E[\bM_{[S]}^{-1}])}\log(\tfrac{\xi(\bx^0)}{\epsilon})$ & $\tfrac{1}{\rho(\bx^0) \lambda_{\min}(\E[\bM_{[S]}^{-1}]) \epsilon}$ &  $\frac{ \xi(\bx^0)}{\lambda_{\min}(\E[\bM_{[S]}^{-1}]) \epsilon} \log(\frac{\xi(\bx^0)}{\epsilon})$ \\ \hline
greedy minibatch & $\tfrac{1}{\mu \lambda_{\min}(\E[\bM_{[S]}^{-1}])}\log(\tfrac{\xi(\bx^0)}{\epsilon})$ & $\tfrac{1}{\rho(\bx^0) \lambda_{\min}(\E[\bM_{[S]}^{-1}]) \epsilon}$ &  $\frac{ \xi(\bx^0)}{\lambda_{\min}(\E[\bM_{[S]}^{-1}]) \epsilon} \log(\frac{\xi(\bx^0)}{\epsilon})$ \\ \hline
\end{tabular}
\caption{Iteration complexity guarantees for $\E[\xi(\bx^K)] \leq \epsilon$ in the smooth case ($g=0$).}
\label{tab:rates_smooth}
\end{table}

\begin{table}[ht]
\centering
\begin{tabular}{|l|c|c|c|}
\hline
{\bf block selection rule} & {\bf strongly PL} &{\bf weakly PL} & {\bf general nonconvex}  \\ \hline \hline
gradient descent & $\tfrac{\lambda_{\max}(\bM)}{\mu}\log(\tfrac{\xi(\bx^0)}{\epsilon})$ & $\tfrac{\lambda_{\max}(\bM)}{\rho(\bx^0) \epsilon}$ &  $\frac{\lambda_{\max}(\bM) \xi(\bx^0)}{\epsilon} \log(\frac{\xi(\bx^0)}{\epsilon})$ \\ \hline
uniform coordinate & $\tfrac{n\max_i\{M_{ii}\} }{\mu}\log(\tfrac{\xi(\bx^0)}{\epsilon})$ & $\tfrac{n \max_i \{M_{ii}\}}{\rho(\bx^0) \epsilon}$ &  $\frac{n \max_i \{M_{ii}\} \xi(\bx^0)}{\epsilon} \log(\frac{\xi(\bx^0)}{\epsilon})$ \\ \hline
greedy coordinate & $\tfrac{n\max_i\{M_{ii}\} }{\mu}\log(\tfrac{\xi(\bx^0)}{\epsilon})$ & $\tfrac{n \max_i \{M_{ii}\}}{\rho(\bx^0) \epsilon}$ &  $\frac{n \max_i \{M_{ii}\} \xi(\bx^0)}{\epsilon} \log(\frac{\xi(\bx^0)}{\epsilon})$ \\ \hline
uniform minibatch & $\tfrac{nL_{\tau}}{\tau\mu}\log(\tfrac{\xi(\bx^0)}{\epsilon})$ & $\tfrac{nL_{\tau}}{\rho(\bx^0) \tau\epsilon}$ &  $\frac{ \xi(\bx^0) n L_{\tau}}{ \tau\epsilon} \log(\frac{\xi(\bx^0)}{\epsilon})$ \\ \hline
greedy minibatch & $\tfrac{nL_{\tau}}{\tau\mu}\log(\tfrac{\xi(\bx^0)}{\epsilon})$ & $\tfrac{nL_{\tau}}{\rho(\bx^0) \tau\epsilon}$ &  $\frac{ \xi(\bx^0) n L_{\tau}}{ \tau\epsilon} \log(\frac{\xi(\bx^0)}{\epsilon})$ \\ \hline
\end{tabular}
\caption{Iteration complexity guarantees for $\E[\xi(\bx^K)] \leq \epsilon$  in the non-smooth case ($g\neq 0$).}
\label{tab:rates_nonsmooth}
\end{table}

To the best of our knowledge, all the rates are novel, except those for strongly PL functions for gradient descent and uniform and greedy coordinate descent, in both smooth (i.e., $g=0$) and non-smooth (i.e., $g\neq 0$)  cases.  These were already shown in \cite{karimi2016linear}. Even in these cases, our class of strongly PL functions is somewhat larger than that considered  in \cite{karimi2016linear} in the non-smooth case ($g\neq 0$).
\end{enumerate}

\subsection{Outline}

We first perform the our general analysis specified for smooth gradient descent in Section~\ref{sec:GD}. In Section~\ref{sec:general_setup} we introduce the general setup considered in the rest of the work. In Section~\ref{sec:proportion} we introduce the proportion function, which is a tool for the general analysis of block selection rules. In Sections~\ref{sec:SPL}, \ref{sec:WPLg} and \ref{sec:general_nonconvex} we establish the main theory for strongly PL, weakly PL, and general non-convex functions, respectively. Finally, in Section~\ref{sec:experiments} we perform numerical experiments confirming our theoretical findings.

\subsection{Notation}

We use boldface to denote a multi-dimensional object. As an example, we have a vector $\bx$, a matrix $\bX$, while a scalar entry of a vector $x_i$ has a normal typeset. By $[n]$ we denote the set $\{1, \dots, n\}$. $\|\bx\| = \langle \bx, \bx \rangle^{1/2} $ is the L2 norm, where $\langle \bx, \by \rangle = \sum_i x_i y_i$ is the standard inner product. refer to Table~\ref{tab:notation} in the appendix for a summary of frequently used notation.


\section{Gradient Descent}\label{sec:GD}

 We assume throughout this section that $f$ is $L$-smooth for some $L>0$: \begin{equation}\label{eq:L-smoothness} 
 f(\bx+\bh) \qq{\leq} f(\bx) + \langle \nabla f(\bx), \bh \rangle + \tfrac{L}{2}\|\bh\|^2, \qquad \bx,\bh\in \R^n.
 \end{equation}  
We shall write $f\in C^1(L)$. In addition to this assumption, in our analysis we consider several  classes of {\em nonconvex} objectives: PL functions, weak PL functions, and gradient dominated functions. 
 
 In this section we perform a novel analysis of the gradient descent\footnote{For simplicity, we consider gradient descent with fixed stepsize inversely proportional to the Lipschitz constant: $1/L$. While one can extend our results to other stepsize strategies using standard techniques, we avoid doing so as to present our results in a simple setting.}  method for minimizing $f$:
\begin{equation}\label{eq:GD}
\bx^{k+1} \qq{=} \bx^k - \tfrac{1}{L}\nabla f(\bx^k).
\end{equation}
for the above classes of nonconvex functions. As we shall show, for these classes of objectives gradient descent converges to the global minimizer. By focusing on the notoriously known gradient descent method first,  we illuminate some of the key insights of this paper without distractions from additional complications caused by the proximal setup and particularities of other algorithms, making the more general treatment in further sections  more easily digestable. 

A key role in the analysis is played by the {\em forcing function} associated with $f$, defined as
\begin{equation}\label{eq:mu-GD-smooth}
\mu(\bx) \qq{\eqdef} \frac{\|\nabla f(\bx)\|^2}{2 \xi(\bx)}, \qquad x\in \R^n/\cX^*.
\end{equation}
For any fixed value of this function, the smaller the  gradient $\|\nabla f(\bx)\|^2$ is, the smaller the optimality gap $\xi(\bx)$. In other words, small gradients force the optimality gap to  become small. The importance of this function is clear from the following simple lemma, which  says that the larger $\mu(\bx^k)$ is, the more reduction we get at iteration $k$ in the optimality gap. 
 
\begin{lemma}\label{lem:firstGD} Let $f\in C^1(L)$ and let $\{\bx^k\}_{k\geq 0}$ be the sequence of iterates produced by the gradient descent method \eqref{eq:GD}. As long as $\bx^k \notin \cX^*$, we have
\[\xi(\bx^{k+1}) \qq{\leq} \left(1-\tfrac{\mu(\bx^k)}{L}\right)\xi(\bx^k).\]
Moreover,  $\|\nabla f(\bx)\|^2 \leq 2L \xi(\bx)$ for all $\bx$.
\end{lemma}
\begin{proof} Let $\bh^k = -\tfrac{1}{L}\nabla f(\bx^k)$. Then
\begin{eqnarray*}
\xi(\bx^{k+1}) &\overset{\eqref{eq:GD}}{=}& \xi(\bx^k + \bh^k) 
\qq{\overset{\eqref{eq:xi}}{=}} f(\bx^k + \bh^k) - f(\bx^*)\\
&\overset{\eqref{eq:L-smoothness}}{\leq} & f(\bx^k) + \langle f(\bx^k), \bh^k \rangle + \tfrac{L}{2}\|\bh^k\|^2 - f(\bx^*) 
\qq{\overset{\eqref{eq:xi}}{=}} \xi(\bx^k) +  \langle f(\bx^k), \bh^k \rangle + \tfrac{L}{2}\|\bh^k\|^2 \\
&=& \xi(\bx^k)- \tfrac{1}{2L}\|\nabla f(\bx^k)\|^2 \qq{\overset{\eqref{eq:mu-GD-smooth}}{=}} (1-\tfrac{\mu(\bx^k)}{L})\xi(\bx^k).
\end{eqnarray*}
Since $\xi(\bx^{k+1})\geq 0$, it must be the case that $\mu(\bx) \leq L$ for all $\bx \notin \cX^*$.
\end{proof}

It is well known that gradient descent is monotonic: $\xi(\bx^{k+1})\leq \xi(\bx^k)$ for all $k$. Note that this property  follows from the second-to-last identity in the proof, and relies on the assumption of $L$-smoothness only.  If $f$ is $\mu$-strongly convex or, more generally, if $f\in \cS_{PL}(\mu)$, then $\mu(\bx) \geq \mu$ for all $\bx \notin \cX^*$, and Lemma~\ref{lem:firstGD} implies the linear rate $\xi(\bx^{k+1}) \leq \left(1-\tfrac{\mu}{L}\right)\xi(\bx^k)$. This result was shown already by Polyak~\cite{polyak1963gradient}.

\subsection{Weakly Polyak-\L{}ojasiewicz functions}

Consider now functions satisfying a weak version of the PL inequality. To the best of our knowledge, this is the first work where such functions are considered.

\begin{definition}[Weak Polyak-\L{}ojasiewicz functions] \label{def:weaklyPL_smooth}
 We say that $f:\R^n\to \R$ is a weak Polyak-\L{}ojasiewicz (WPL) function with parameter $\mu\geq 0$ if there exists $\bx^*\in \cX^*$ such that 
 \begin{equation} \label{eq:weak_PL_def} 
 \sqrt{\mu} \cdot \xi(\bx) \qq{\leq}  \|\nabla f(\bx)\| \cdot \|\bx-\bx^*\|, \qquad \bx\in \R^n.
\end{equation}
For simplicity, we write $f\in \cW_{PL}(\mu)$.
\end{definition}

Consider the Huber loss given by \[H(z) \qq{\eqdef} \begin{cases}
z^2 \quad &|z| < 1 \\
2|z| - 1 \quad & \mbox{otherwise},
\end{cases}\]
and the derived function $f$ given by $f(x_1, x_2) = H(x_1)H(x_2)$. It is straightforward to show from the definition that $f$ is smooth function for which $f \in \cW_{PL}(\mu)$ for some $\mu > 0$, while $f \notin \cS_{PL}(\mu)$ for all $\mu > 0$.

Note that all\footnote{By ``all'' we implicitly mean all functions for which the definition make sense. That is, differentiable and having a global minimizer $\bx^*$.} functions belong to $\cW_{PL}(0)$. As the next result shows, WPL functions admit a  lower bound on $\mu(\bx)$ which is proportional to $\xi(\bx)$ and inversely proportional to $\|\bx-\bx^*\|^2$. 

\begin{lemma}\label{lem:nhd89hg30did} If $f\in \cW_{PL}(\mu)$, then
\[\mu(\bx) \qq{\geq}  \frac{\mu \xi(\bx)}{2\|\bx-\bx^*\|^2}, \qquad \bx\in \R^n/\cX^*.\]
\end{lemma}
\begin{proof} We have
\[
\mu(\bx) \qq{\overset{\eqref{eq:mu-GD-smooth}}{=}}  \frac{\|\nabla f(\bx)\|^2}{2 \xi(\bx)} \qq{\overset{\eqref{eq:weak_PL_def} }{\geq}}  \frac{\mu \xi^2(\bx)/\|\bx-\bx^*\|^2}{2\xi(\bx) }  \qq{=}  \frac{\mu \xi(\bx)}{2\|\bx-\bx^*\|^2}.
\]
\end{proof}

Several basic properties of WPL functions are summarized in Appendix~\ref{Appendix-WPL}. Combining Lemma~\ref{lem:firstGD} and Lemma~\ref{lem:nhd89hg30did}, we get the recursion

\begin{equation}\label{eq:recursion}
\xi(\bx^{k+1}) \qq{\leq} \left( 1-  \frac{\mu \xi(\bx^k)}{2L\|\bx^k-\bx^*\|^2} \right) \xi(\bx^k).
\end{equation}

The next lemma will be useful in the analysis of this recursion.

\begin{lemma} \label{lem:onestepbound_xi}
Let $\{\alpha^t\}_{t=0}^k$ and $\{\beta^t\}_{t=0}^k$ be two sequences of positive numbers satisfying the recursion \begin{equation} \label{eq:lem:onestepbound}
\alpha^{t+1} \qq{\leq} \left(1 - \alpha^t \beta^t\right)\alpha^t.
\end{equation} Then for all $k\geq 0$ we have the bound \[\alpha^k \qq{\leq} \frac{\alpha^0}{1 + \alpha^0\sum_{t=0}^{k-1}\beta^t}.\]
\end{lemma}
\begin{proof}
As $\alpha^t$ and $\beta^t$ are positive numbers, we have $\alpha^{t+1} \leq \alpha^t$ for all $t$ using \eqref{eq:lem:onestepbound}. Observe that \[\frac{1}{\alpha^{t+1}} - \frac{1}{\alpha^t} \qq{=} \frac{\alpha^t - \alpha^{t+1}}{\alpha^{t+1}\alpha^{t}} \qq{\geq} \frac{\alpha^{t} - \alpha^{t+1}}{(\alpha^t)^2} \qq{\stackrel{\eqref{eq:lem:onestepbound}}{\geq}} \beta^t,\]
which we can recursively used to show \[\frac{1}{\alpha^{k}} \qq{\geq} \frac{1}{\alpha^{k-1}} + \beta^{k-1} \qq{\geq} \dots \qq{\geq} \frac{1}{\alpha^{0}}+ \sum_{t=0}^{k-1}\beta^t.\]
We get the result by inverting the last equation.
\end{proof}

By applying the above lemma to recursion \eqref{eq:recursion}, we get a global convergence result for gradient descent applied to a WPL function. 

\begin{theorem} Let $f\in \cC^1(L)$, and let $\{\bx^k\}_{k\geq 0}$ be the sequence of iterates produced by the gradient descent method \eqref{eq:GD}. Assume $f\in \cW_{PL}(\mu)$ for $\mu>0$. Then for  all $k\geq 1$ we have
\begin{equation}\label{eq:bound98698y}
\xi(\bx^{k}) \qq{\leq} \frac{\xi(\bx^0)}{1+\xi(\bx^0) \tfrac{\mu}{2L}\sum_{t=0}^{k-1} \tfrac{1}{\|\bx^t-\bx^*\|^2}} \qq{\leq} \frac{2L}{\mu}\cdot \frac{1}{\sum_{t=0}^{k-1} \frac{1}{\|\bx^t-\bx^*\|^2}} . 
\end{equation}
\end{theorem}

The second inequality is obtained from the first by neglecting the additive constant 1 in the denominator.

 By monotonicity, all iterates of gradient descent stay in the level set $\cL_0\eqdef \{\bx\in \R^n \;:\; f(\bx) \leq f(\bx^0)\}$. If  this set is bounded, then $R\eqdef  \max_{\bx\in \cL_0} \|\bx-\bx^*\| < +\infty$, and we have $\|\bx^k-\bx^*\| \leq R$ for all $k$. In this case, the bound \eqref{eq:bound98698y} implies
\[\xi(\bx^k) \qq{\leq} \frac{2L R^2}{\mu k}.\]

\subsection{Gradient dominated functions}

We now consider a new class of (not necessarily convex) functions.  To the best of our knowledge, this class was not considered in optimization before.

\begin{definition} We say that function $f:\R^n\to \R$ is {$\varphi$-gradient dominated} if there exists a  function  $\varphi:\R_+ \to \R_+$ such that $\varphi(0)=0$, $\lim_{t\to 0} \varphi(t) = 0$ and  \begin{equation} \label{eq:nd09un09s} \xi(\bx) \qq{\leq} \varphi(\|\nabla f(\bx)\|), \qquad \bx\in \R^n.\end{equation}
\end{definition}

The above definition essentially says that for any sequence $\{\bx^k\}$ (not necessarily related to iterates of gradient descent) such that $\|\nabla f(\bx^k)\|\to 0$, we must have $f(\bx^k)\to f(\bx^*)$. In particular, if $f$ has multiple minimizers, all must have the same function value.

As an example of the function $\varphi$, we might consider any function of the form $\varphi(t) = c \cdot |t|^p$, where $c > 0$ and $p > 0$. The specific choice of $p = 2$ was already considered before in \cite{reddi2016stochastic}.

\begin{theorem} Assume that $f\in \cC^1(L)$  is $\varphi$-gradient dominated.
Pick $\epsilon>0$ and let 
\begin{equation}\label{eq:K} 
k\qq{\geq} \frac{2L \xi(\bx^0)}{\epsilon}\log\left(\frac{\xi(\bx^0)}{\varphi(\epsilon)}\right).
\end{equation} Then
$\min \{\xi(\bx^t)  \;:\; t=0,1,\dots,k\} \leq \varphi(\epsilon)$.
\end{theorem}
\begin{proof} If $\|\nabla f(\bx^t)\|^2 \leq \epsilon$ for some $k=0,1,\dots,k-1$, we are done by applying \eqref{eq:nd09un09s}. Otherwise, $\|\nabla f(\bx^t)\|^2 > \epsilon$ for all $t=0,1,\dots,k-1$. Then in view of \eqref{eq:mu-GD-smooth}, $\mu(\bx^t) > \epsilon/(2\xi(\bx^t))$ for all such $t$. By monotonicity, $\xi(\bx^{t+1}) \leq \xi(\bx^t)$ for all $t$, whence $\mu(\bx^t) > \epsilon/(2\xi(\bx^0)).$ 
By applying Lemma~\ref{lem:firstGD}, we get 
\[\xi(\bx^{t+1}) \qq{\leq} \left(1-\frac{\epsilon}{2L\xi(\bx^0)}\right)\xi(\bx^t), \qquad t=0,1,\dots,k-1.\] 
By unrolling the recurrence, and using the bound $(1- s)^{1/s} \leq \exp(-1)$ (which holds for $0< s \leq 1$),  we get
\[\xi(\bx^{k}) \qq{\leq} \left(1-\frac{\epsilon}{2L\xi(\bx^0)}\right)^k\xi(\bx^0) \qq{\leq} \exp\left(-\frac{\epsilon k}{2L\xi(\bx^0)}\right)\xi(\bx^0) \qq{\overset{\eqref{eq:K}}{\leq}} \varphi(\epsilon).\]
\end{proof}

\subsection{Brief literature review}

The original gradient descent method was developed by Cauchy \cite{cauchy1847methode} and it has seen a lot of development ever since. This development is documented in detail in \cite{nesterov2013introductory, bubeck2015convex}. In the recent years, a version of gradient descent called coordinate descent was developed. The first developments of coordinate descent are due to \cite{leventhal2010randomized} and it was first analyzed for general convex objectives by Nesterov in \cite{Nesterov:2010RCDM}.

An important part of the coordinate descent is its ability to work with arbitrary block selection strategies. In the seminal work of Nesterov \cite{Nesterov:2010RCDM}, there were three strategies introduced, which are known as coordinate descent with uniform probabilities, coordinate descent with importance sampling, and greedy coordinate descent. The first two strategies fall into the family of randomized strategies. These were further developed in \cite{UCDC, NSync, PCDM, ESO, AdaSDCA, SDNA, ISM}. The third selection rule is a deterministic strategy similar in nature to batch gradient descent and it was shown to be superior over randomized methods in terms of iteration complexity in \cite{nutini2015coordinate}.

The  Polyak-\L{}ojasiewicz condition was first introduced by Boris Polyak in \cite{polyak1963gradient}. It was revived recently in \cite{karimi2016linear} and applied to modern optimization approaches. Since then, multiple papers used the condition to develop new approaches \cite{de2016big, lei2017nonconvex, gao2016ojasiewicz}. Gradient dominated functions were recently considered in \cite{reddi2016stochastic}.

Lastly, we note that our framework considers the non-accelerated version of coordinate descent methods, although they play a key role in modern theory. If needed, the acceleration can be achieved non-directly by using approaches as proposed in \cite{lin2015universal} or \cite{frostig2015regularizing}. We leave the accelerated counterpart of this framework to future work. 

\section{General Setup} \label{sec:general_setup}

In this section we  move beyond the simplified setup considered in the previous section and introduce the setting considered in this paper in its full generality. Our general treatment differs from that in Section~\ref{sec:GD} in several ways. 

First, we consider the {\em composite} optimization problem \begin{equation} \label{eq:prox_problem}
\min_{x\in \R^n} \left\{ F(\bx) \qq{\eqdef} f(\bx) + g(\bx)\right\}, 
\end{equation}
where $f$ is assumed to be smooth, and $g$ is a simple (possibly nonconvex and nonsmooth) separable function. Second, we $f$ is assumed to be smooth in a slightly more general sense than $L$-smoothness of Section~\ref{sec:GD}. Third, we go beyond gradient descent and consider a large family of first order methods which include randomized, cyclic, adaptive and greedy coordinate descent, in serial and adaptive settings. 

In the following, we will refer to the problem \eqref{eq:prox_problem} with $g = 0$ as the smooth case and otherwise as the non-smooth case.

\subsection{Smoothness and separability} \label{sec:smoothness}

We assume, that $f$ is $\bM$-smooth, which is formalized by the following assumption:
\begin{assumption} \label{ass:Msmoothness}
We say that a function $f:\R^n \rightarrow \R$ is $\bM$-smooth, if there exists a positive definite matrix $\bM \in \R^{n \times n}$ such that \begin{equation} \label{eq:Msmoothness}
f(\bx + \bh) \qq{\leq} f(\bx) + \langle \nabla f(\bx), \bh \rangle + \tfrac{1}{2} \langle \bM \bh, \bh \rangle, \qquad  \bx,\bh \in \R^n.
\end{equation}
\end{assumption}
In the non-smooth case ($g \neq 0$), we will without loss of generality assume $\bM$ to be a multiple of the identity matrix; specifically $\bM := L \bI$,  where $\bI$ is the identity matrix. If \eqref{ass:Msmoothness} holds for some $\bM$, we can always replace $\bM$ by $L\bI$,   where $L = \lambda_{\max}(\bM)$. It is easy to verify, that if a function is $\bM$-smooth, it is also $L\bI$-smooth. Note that in some cases we might choose $L$ to be smaller. This will be in detail explained in Section~\ref{sec:proportion_nonsmooth}.

For simplicity, we will use the notation with $\bM$ also in the non-smooth case, but we will always treat it as the diagonal matrix $L \bI$.

In addition to smoothness of $f$, we assume that the function $g$ is separable, which is defined as follows:
\begin{assumption} \label{ass:separability}
We say that a function $g: \R^n \rightarrow \extR \eqdef \R \cup \{+\infty\}$ is separable, if there exist $n$ scalar functions $g_1, \dots, g_n:\R^n\to \extR $, such that \begin{equation} \label{eq:separability}
g(\bx) \qq{=} \sum_{i=1}^n g_i(x_i).
\end{equation}
\end{assumption}
Note that the function $g$ is treated as the non-smooth part of the problem (e.g., L1 norm, box constraints, and so on). When we refer to a smooth problem, we assume the setup with $g = 0$, while all the other setups are referred as non-smooth problems.

The problem described in \eqref{eq:prox_problem} is encountered in many areas, ranging from machine learning and signal processing to biology and beyond. We believe it does not need to be motivated further, as it was already considered in a lot of previous works.

\subsection{Masking vectors and matrices}

 Let $\bx \in \R^n$ be an arbitrary vector and let $\bX \in \R^{n\times n}$ be an arbitrary matrix. We will need to index vectors and matrices by subsets of coordinates $\emptyset \neq S \subseteq [n]$. The indexing has two distinct forms. By $\bx_S$ we denote the $|S|$-dimensional vector constructed by taking the entries of $\bx$ with indices in $S$, while the notation $\bx_{[S]}$ is used to zero out every entry of $\bx$ not appearing in $S$ without changing its length. We have similar notation for matrices, where $\bX_{S}$ denotes the $|S| \times |S|$ matrix of the entries with both column and row indices in $S$, while $\bX_{[S]}$ is used for the matrix $\bX$ with entries zeroes out outside of columns and rows with indices in $S$.  As a rule of thumb, the subscript $S$ changes the dimensions of the object, while the subscript $[S]$ maintains its dimensions.
 
 To illustrate the notation,  consider the following example.
 
 \begin{example}  Let $\bx\in \R^3$, $\bX\in \R^{3\times 3}$ and $S = \{1,3\}$. Then 
\[\bx = \left[\begin{array}{c}
1 \\ 2 \\ 3
\end{array} \right]  \qquad \Rightarrow \qquad \bx_{S} = \left[ \begin{array}{c}
1 \\ 3\end{array}\right], \qquad \bx_{[S]} = \left[ \begin{array}{c}
1 \\ 0 \\ 3\end{array}\right] \]
and
\[\bX = \left[ \begin{array}{ccc}
1 & 4 & 7 \\
2 & 5 & 8 \\
3 & 6 & 9
\end{array}\right] \qquad \Rightarrow \qquad \bX_{S} = \left[ \begin{array}{cc}
1 & 7 \\
3 & 9
\end{array}\right], \qquad \bX_{[S]}= \left[ \begin{array}{ccc}
1 & 0 & 7 \\
0 & 0 & 0 \\
3 & 0 & 9
\end{array}\right].\]
\end{example}

\subsection{Algorithm}

We now propose and analyze a wide class of block descent algorithms for solving \eqref{eq:prox_problem}. For a non-empty block of coordinates $S \subseteq [n]$, we define\begin{equation} \label{def:US}
U_S(\bx, \bu) \qq{\eqdef} \langle (\nabla f(\bx))_{[S]}, \bu \rangle + \frac{1}{2}\bu^\top \bM_{[S]} \bu + \sum_{i \in S} \left[ g_i(x_i + u_i) - g_i(x_i) \right].
\end{equation}

We assume that finding a minimizer of $U_S(\bx, \bh)$ in $\bh$ is cheap (e.g., there exists a closed form solution, or an efficient algorithm). Given an iterate $\bx^k$, in iteration $k$ of our method we select a block $S_k\subseteq [n]$ of active coordinates, according to an {\em arbitrary block selection procedure} $\cP_k$, and subsequently minimize $U_{S_k}(\bx^k,\bh)$ in $\bh$. The result is denoted $\bu^k$.  Due to the structure of problem \eqref{def:US}, the minimizer of $U_{S_k}$ does not depend on $u_i$ for $i\notin S_k$. Hence, only the active coordinates  $\bu^k_i$ for  $i\in S_k$ are relevant, and we set $\bx^{k+1} = \bx^k + \bu^k_{[S_k]}$. Equivalently, using other alternative notation, we can write this as
\[\bx^{k+1}_i \qq{=} \begin{cases} \bx^k_i + \bu^k_{i}, &\qquad i\in S_k,\\
\bx^k_i, & \qquad i\notin S_k,\end{cases}\]
or
\[\bx^{k+1}_{S_k} = \bx^k_{S_k} + \bu^k_{S_k}, \qquad \bx^{k+1}_{[n]/S_k} = \bx^k_{[n]/S_k}.\]

This is Algorithm~\ref{alg:general}.

\begin{algorithm}
\caption{Proximal Arbitrary-Block Descent Method}
\label{alg:general}
\begin{algorithmic}
\STATE \textbf{Input:} Initial iterate $\bx^{0}\in {\rm dom}( g)$, arbitrary block selection procedures $\{\mathcal{P}_k\}_{k=0}^{K-1}$
\FOR{$k = 0, \dots, K-1$}
\STATE Pick a non-empty subset of coordinates $S_k \subset [n]$ using the procedure $\mathcal{P}_k$
\STATE Compute $\bu^k \in \argmin_{\bu} \{U_{S_k}(\bx^{k}, \bu)\}$
\STATE Update $\bx^{k+1} = \bx^{k} + \bu^k_{[S_k]}$
\ENDFOR
\STATE Output $\bx^K$
\end{algorithmic}
\end{algorithm}
Observe, that in the smooth case we have \begin{equation} \label{eq:US_min_smooth}
\bu^k_{S_k} \qq{=} - \bM_{S_k}^{-1} \nabla_{S_k} f(\bx^k),
\end{equation}
where $\nabla_S f(\bx) \eqdef (\nabla f(\bx))_S$. The iteration can be computed as a solution of a linear system of a size $|S_k| \times |S_k|$, which is very cheap for small $|S_k|$. In the non-smooth case, the iterate does not have a closed-form solution in general  
but can be solved fast for a lot of forms of $g$, e.g., box constraints or L1 norm. Since we assume $\bM=L\bI$, we can write
\begin{equation} \label{eq:US_min_nonsmooth}
\bu^k_{S_k} \qq{=} \argmin_{\bu \in \R^{|S_k|}} \bigg\{ \langle \nabla_S f(\bx^k), \bu \rangle + \frac{L}{2}\|\bu\|^2 + \sum_{i \in S_k} g_i(x_i + u_i) \bigg\}.
\end{equation}

For simplicity, assume $\bM = L \bI$ for some constant $L>0$. If  $g=0$,  and we always pick $S_k = [n]$, then Algorithm~\ref{alg:general} reduces to gradient descent, considered in Section~\ref{sec:GD}. If $g\neq 0$, and we always pick $S_k = [n]$, then Algorithm~\ref{alg:general} reduces to proximal gradient descent. On the other hand, if we always pick $|S_k|=1$, then we obtain coordinate descent ($g=0$) or proximal coordinate descent ($g\neq 0$). The selection procedure $\cP_k$ may be set to choose the coordinates in a cyclic manner, greedily, randomly according to any (fixed or  evolving) probability law, and even adaptively to the entire history of the iterative process. There are many other possibilities between the two extremes of always selecting $|S_k|=1$ and $|S_k|=n$. Such methods can be considered block coordinate descent methods, subspace descent methods, or parallel coordinate descent methods (as the updates to individual coordinates can  be performed in parallel). We stress that unlike all other methods considered in the literature,  in our method we allow for the block selection procedure $\cP_k$ to be {\em arbitrary}, without any restrictions whatsover. 

By removing these restrictions, we allow for several new possibilities in the sampling procedure. These are: 1) the block selection procedure might change from iteration to iteration, allowing for adaptive strategies as \cite{AdaSDCA}, 2) the procedure might depend on previous iterations, which opens up the possibility of cyclic and other similar selections, and 3) the procedure does not have to be randomized, which allows for greedy selection procedures \cite{nutini2015coordinate}.

In all our convergence results we shall enforce several key common assumptions, together with some additional assumptions. In order to avoid repeating the common core, we shall summarize them here.

\begin{assumption}[Common Assumptions] \label{ass:non-convex_setup}
Let $f : \R^n \rightarrow \R$ be an $\bM$-smooth \eqref{eq:Msmoothness} function, let $g:\R^n \rightarrow \extR$ be separable \eqref{eq:separability}, and let $F: \R^n \rightarrow \extR$ be function defined using $f$ and $g$ as in \eqref{eq:prox_problem}. Assume $F$ has a global minimizer  $\bx^*$, such that $F(\bx^*) > -\infty$. Let $\bx^0\in {\rm dom}(g)$ be an initial point and let the sequence $\{\bx^k\}_{k=1}^K$ be generated using Algorithm~\ref{alg:general}, where  $\{S_k\}_{k=0}^{K-1}$ is an arbitrary sequence of non-empty (possibly random) subsets of $[n]$.
\end{assumption}

\subsection{Forcing function}

As before, let  us define the {\em optimality gap function}
\begin{equation} \label{eq:optimality_gap}
\xi(\bx) \qq{\eqdef} F(\bx) -  F(\bx^*),
\end{equation}
where $\bx^*$ is a minimizer of $F$. Observe that $\xi(\bx) \geq 0$ with equality only if $\bx \in \cX^* \eqdef \argmin_{\bx} F(\bx)$. 

We now extend the definition of the forcing function to the proximal setting. 

\begin{definition}[Forcing function: proximal version] \label{def:GPL_function}
Let 
\begin{equation} \label{eq:lambda_ih8d} 
\lambda(\bx) \qq{\eqdef} - L \cdot \min_{\by \in \R^n}\left\{ \langle \nabla f(\bx), \by \rangle + \frac{L}{2}\|\by\|^2 + g(\bx + \by) - g(\bx) \right\},
\end{equation} 
where $L$ is chosen such that \eqref{eq:Msmoothness} holds with $\bM = L\bI$. Specific choices of $L$ are explained in detail in Section~\ref{sec:proportion_nonsmooth}.
The non-negative function  
\begin{equation} \label{eq:GPL_function}
\mu(\bx) \qq{\eqdef}  \frac{\lambda(\bx)}{\xi(\bx)}  
\end{equation} is the {\em proximal forcing function}.  The domain of $\bx$ is ${\rm dom}(g) / \cX^*$. 
\end{definition}

Pick any $\bx\in {\rm dom}(g)/\cX^*$. The minimum in \eqref{eq:lambda_ih8d} is non-positive since setting $\by = \mathbf{0}$ gives a zero value. Hence, $\lambda(\bx)\geq 0$. Since the denominator in \eqref{eq:GPL_function} is   positive, $\mu(\bx)$ is always nonnegative. 

 In the smooth case ($g = 0$) we have \begin{equation} \label{eq:GPL_function_smooth}
\mu(\bx) \qq{=} \frac{\|\nabla f (\bx)\|^2}{2 \xi(\bx)},
\end{equation}
which is simply the forcing function \eqref{eq:mu-GD-smooth}.

\subsection{Proportion function}
We introduce one more notion, which we call the \textit{proportion function}. This function plays an important role in our theory.

\begin{definition} \label{def:proportion_function}
Let \begin{equation} \label{eq:set_Xalpha_proportion}
\cX \qq{\eqdef} \left\{\bx \in \R^n ~:~ \lambda(\bx) \neq 0 \right\}.
\end{equation} 
The \textit{proportion function} is defined by \begin{equation} \label{eq:proportion_function}
\theta(S,\bx) \qq{\eqdef} \frac{- \min_{\bu \in \R^n} \bU_S(\bx,\bu)}{\lambda(\bx)}
\end{equation}
for all $S \in \cS$ and $\bx \in \cX$. For $\bx \in \cX$, we set $\theta(S, \bx) = 0$. 
\end{definition}
Similarly as in the definition of the forcing  function, both of the minimizations in \eqref{eq:proportion_function} are non-positive, as $\bu = 0$ and $\by = 0$ gives zero value in the numerator and denominator, respectively.
Also observe that in the smooth case ($g = 0$) we have \begin{equation} \label{eq:proportion_function_smooth}
\theta(S,\bx) \qq{=} \frac{ (\nabla_S f(\bx))^\top (\bM_{S})^{-1} \nabla_S f(\bx)}{\|\nabla f(\bx)\|^2}.
\end{equation}
Note that the matrix $(\bM_{S})^{-1}$ exists since all principal submatrices of a positive definite matrix are also positive definite. A more detailed treatment of the proportion function can be found in Section~\ref{sec:proportion}.

Also note that $\cX$ \eqref{eq:set_Xalpha_proportion} might differ from $\cX^*$ in the case of local minimizers.

\subsection{Generic descent lemma}

We now formulate a simple but  important descent lemma which bounds the progress gained by a single iteration of Algorithm~\ref{alg:general}. Our bound applies to arbitrary block selection rules, and  will enable us to prove global convergence results for Algorithm~\ref{alg:general} for new classes of nonconvex functions.

\begin{lemma} \label{lem:onestepbound}

Let $\bx^{k+1}$ be the next iterate of Algorithm~\ref{alg:general} generated from $\bx^k$ by picking a nonempty set of coordinates $S_k\subseteq [n]$. Then  \begin{equation} \label{eq:onestepbound}
\xi(\bx^{k+1}) \qq{\leq} \big[1 - \theta(S_k,\bx^k) \cdot \mu(\bx^k)\big] \cdot \xi(\bx^k).
\end{equation}
Applying this repeatedly, for all $K\geq 1$ we obtain the estimate
 \begin{equation} \label{eq:onestepbound_global}
\xi(\bx^K) \qq{\leq} \left(\prod_{k=0}^{K-1}\left[1 -  \theta(S_k,\bx^{k}) \cdot \mu(\bx^k)\right]\right) \xi(\bx^0).
\end{equation}
\end{lemma}
\begin{proof}
See Section~\ref{sec:proof_onestepbound}
\end{proof}

Recursion \eqref{eq:onestepbound} is a direct generalization of Lemma~\ref{lem:firstGD}. Indeed, if $\bM =L \bI$, $g=0$, and $S_k=[n]$, then in the view of \eqref{eq:proportion_function_smooth}, $\theta([n],\bx^k) = 1/L$. Note that as $\theta \geq 0$ and $\mu \geq 0$, we can be sure that $\xi(\bx^{k+1}) \leq \xi(\bx^{k})$, which means that we are not getting worse by iterating the Algorithm~\ref{alg:general}. The proof of Lemma~\ref{lem:onestepbound} is straightforward. The difficulty will lie in bounding the forcing and proportion functions so as to obtain convergence.   In Sections~\ref{sec:SPLg}, \ref{sec:WPLg} and \ref{sec:general_nonconvex} we apply this lemma to prove the main results of this paper. 

The $K$-step bound \eqref{eq:onestepbound_global} provides us with a compact and generic bound on the optimality gap at the $K$-th iterate of Algorithm~\ref{alg:general} dependent on the iterates $\bx^0, \bx^1,\dots, \bx^{K-1}$ and the selected sets $S_0,S_1,\dots,S_{k-1}$. This result does not immediately imply convergence as at this level of generality, it is possible for the product appearing in \eqref{eq:onestepbound_global} not to converge to zero. Indeed, this corollary also covers the situation where $S_k = {1}$ for all $k$, which clearly can't result in convergence. We will need to introduce further restrictions in order to establish convergence.

\section{Proportion Function} \label{sec:proportion}

In this section we show standard bounds on the proportion function, which are independent of a given iterate. This will be important further, to recover the convergence rates given by standard theory. We note, that for stochastic methods we bound the expectation of the proportion function conditioned on the last iterate, instead of directly bounding the proportion function. This quantity will be important in the theory for the convergence of stochastic methods, as specified in following sections.

The proportion function is the only quantity in the convergence theory, which is dependent on the choice of the set of coordinates $S$. Therefore, to analyze a new sampling strategy for coordinate descent, one only has to show a bound on the proportion function. This opens up a possible venue of novel techniques for coordinate selection.

In the following we tackle all the known cases of samplings, which we break down by their smoothness. Also, we introduce and analyze a new sampling procedure to showcase the generality of our framework.
\subsection{Smooth problems} \label{sec:proportion_smooth}

In the case of $g = 0$, we have the proportion function equal to \eqref{eq:proportion_function_smooth}, i.e., \[\theta(S,\bx) \qq{=} \frac{ (\nabla_S f(\bx))^\top (\bM_{S})^{-1} \nabla_S f(\bx)}{\|\nabla f(\bx)\|^2}\]
for all $\bx \in \cX^*$, and all $\emptyset \neq S \subseteq [n]$.
Let us break-down the cases according to specific choices of the set $S$.
\begin{itemize}
\item \textbf{Batch Gradient descent:}  In the case when $S = [n]$, we recover the standard gradient descent strategy, which dates back to the work of Cauchy \cite{cauchy1847methode}. In this case we can lower bound the proportion function by \begin{equation} \label{eq:prop_lower_bound_smooth_batch}
\theta([n], \bx) \qq{=} \frac{(\nabla f(\bx))^\top \bM^{-1} \nabla f(\bx)}{\|\nabla f (\bx)\|^2} \qq{\geq} \frac{1}{\lambda_{\max}(\bM)}
\end{equation}
for all $\bx \in \cX^*$. 
\item \textbf{Serial Coordinate Descent:} Suppose $S = \{i\}$, for some given $i$. In this case we get \begin{equation} \label{eq:prop_lower_bound_rcd}
\theta(\{i\}, \bx) \qq{=} \frac{ (\nabla_i f(\bx))^2}{M_{ii} \|\nabla f(\bx)\|^2.}
\end{equation}
There are multiple strategies for choosing the coordinate $i$, and we tackle them one by one.
\begin{itemize}
\item \textbf{Uniform probabilities:} Suppose we choose the coordinate $i$ with the probability $p_i$ given by $p_i = 1/n,$ independently of $\bx$. This strategy was originally analyzed in \cite{Nesterov:2010RCDM}. The expectation of $\theta$ can be lower bounded as \begin{equation} \label{eq:prop_lower_bound_rcd_uniform}
\E_i[\theta(\{i\}, \bx) ~|~ \bx] \qq{\stackrel{\eqref{eq:prop_lower_bound_rcd}}{=}} \frac{1}{n}\sum_{i=1}^n  \frac{(\nabla_i f(\bx))^2}{M_{ii} \|\nabla f(\bx)\|^2} \qq{\geq} \frac{1}{n \cdot \max_i \{M_{ii}\}}.
\end{equation} 
\item \textbf{Importance sampling:} Suppose we choose the coordinate $i$ with the probability $p_i$ given by 
\begin{equation}\label{eq:imp_sampling} p_i \qq{=} \frac{M_{ii}}{\sum_{i=1}^n M_{ii}},\end{equation}
independently of $\bx$. Again, this strategy was originally analyzed in \cite{Nesterov:2010RCDM}. The expectation of $\theta$ can be lower bounded as \begin{equation} \label{eq:prop_lower_bound_rcd_importance}
\E_i[\theta(\{i\}, \bx) ~|~ \bx] \qq{\stackrel{\eqref{eq:prop_lower_bound_rcd}}{=}} \frac{1}{\sum_{i=1}^n M_{ii}}\sum_{i=1}^n  \frac{M_{ii}(\nabla_i f(\bx))^2}{M_{ii}\|\nabla f(\bx)\|^2} \qq{\geq} \frac{1}{ \sum_{i=1}^n M_{ii}}
\end{equation}
\item \textbf{Greedy choice:} Suppose we choose the coordinate $i$ deterministically as \begin{equation} \label{eq:greedy_rule_smooth}
i \qq{=} \argmax_{j} \left\{ \frac{(\nabla_i f(\bx))^2 }{M_{ii}}\right\}. 
\end{equation}
It is straightforward that this strategy maximizes the proportion function \eqref{eq:prop_lower_bound_rcd} for a single iteration, given that we choose only a single coordinate. It was originally proposed in \cite{Nesterov:2010RCDM} and further improved in \cite{nutini2015coordinate}. However, in this case we do not have a better bound than \eqref{eq:prop_lower_bound_rcd_importance}, which would be independent of $\bx$\footnote{In \cite{nutini2015coordinate} they proved a slightly better bound using $\ell_1$-strong convexity, which can be achieved by replacing the $\ell_2$-norm by $\ell_{\infty}$-norm in the definition of the proportion and forcing functions.}. We can get this bound using that the maximum of some quantity is more than its average weighted by $M_{ii}$ \begin{equation} \label{eq:prop_lower_bound_rcd_greedy}
\theta(\{i\}, \bx) \qq{=} \frac{\max_j \{(\nabla_j f(\bx))^2/M_{jj}\}}{\|\nabla f(\bx)\|^2} \qq{\geq} \frac{1}{\sum_{j=1}^n M_{jj}}.
\end{equation}
Observe that in the case of $\|\nabla f(\bx)\|^2 = (\nabla_i f(\bx))^2$ we have that the above lower bound \eqref{eq:prop_lower_bound_rcd_greedy} could be $\sum_{j=1}^n M_{jj}/M_{ii}$ larger to still hold. Therefore on some iterations, greedy coordinate descent is much better than coordinate descent with importance sampling, although their global bounds are the same. This usally leads to superiority of greedy rules in practice, in the case that they can be implemented cheaply.
\end{itemize}
\item \textbf{Minibatch Coordinate Descent:} Suppose $|S| = \tau$, for some given $n > \tau > 1$. There are currently two analyzed strategies in this case and we introduce a third. 
\begin{itemize} 
\item \textbf{$\tau$-nice sampling:} Assume that we want to sample a subset of $\tau > 1$ coordinates at each iteration, uniformly at random from all subsets of cardinality $\tau$. It can be inferred from results established  in \cite{SDNA} that the expectation of the proportion function can be lower bounded by the quantity
\begin{equation} \label{eq:prop_sdna_sampling}
\E[\theta(S, \bx) ~|~ \bx] \qq{\stackrel{\eqref{eq:proportion_function_smooth}}{=}} \frac{ (\nabla f(\bx))^\top \E[\bM_{[S]}^{-1}] \nabla f(\bx)}{\|\nabla f(\bx)\|^2}  \qq{\geq}  \lambda_{\min}\left( \E\left[\bM_{[S]}^{-1}\right]\right),
\end{equation}
where $\bM_{[S]}^{-1}$ is the $n \times n$ matrix constructed by putting the matrix $\bM_{S}^{-1}$ on the columns and rows specified by $S$ and zero out the rest of the entries.
Additionally, assuming that we have a factorization $\bM = \bA^\top \bA$, where $\bA \in \R^{m \times n}$, using results from \cite{ESO} , this can be further lower bounded as \begin{equation} \label{eq:prop_taunice_sampling}
\lambda_{\min}\left( \E\left[\bM_{[S]}^{-1}\right]\right) \qq{\stackrel{\text{\cite{SDNA}}}{\geq}}   \frac{1}{n \cdot \max_i \{v_i\}}, 
\end{equation} 
where  \[v_i \qq{\stackrel{\text{\cite{ESO}}}{\eqdef}} \sum_{j=1}^m \left[ 1 + \frac{(\|\bA_{j:}\|_0 - 1)(\tau - 1)}{n-1} \right] A_{ji}^2.\]

\item \textbf{Importance sampling for minibatches:} A minibatch version of importance sampling was recently proposed in \cite{ISM}. The main idea is as follows: we randomly partition the coordinates into $\tau$ approximately equally sized ``buckets'', and subsequently and independently perform  standard importance sampling (as described above) for each bucket. The  sampling is then generated as the union of all sampled coordinates. For specific bounds, we recommend discussing the original paper \cite{ISM}, as they are not available in a compact form.
\item \textbf{Greedy minibatches:} To showcase the power of our framework, we introduce a brand-new selection rule called \textit{greedy minibatches}. This selection rule aims to select the set $S$ such that it minimizes the proportion function on the current iteration, i.e., \begin{equation} \label{eq:greedy_rule_minibatch_smooth}
S^g \qq{\eqdef} \argmax_{S : |S| = \tau} \left\{ (\nabla_S f(\bx))^\top (\bM_{S})^{-1} \nabla_S f (\bx) \right\}.
\end{equation}
The above selection procedure is a difficult problem in general, but it might be feasible for some specific problems, e.g., for diagonal $\bM$ or the function $f$ with a special structure (see \cite{nutini2015coordinate} for examples).

To get a lower bound on the proportion function independent of the current iterate $\bx$, we use the argument that the maximum of some quantities is always at least equal to any weighted mean of the same quantities. Using this, we get that $\theta(S^g, \bx) \geq \E[\theta(S, \bx) ~|~ \bx]$, where the sets $S$ were selected according to the $\tau$-nice sampling, importance sampling for minibatches introduced above, or any other sampling. Therefore, using \eqref{eq:prop_sdna_sampling} we get that 
\begin{equation} \label{eq:prop_minibatch_greedy}
\theta(S^g, \bx) \qq{\stackrel{\eqref{eq:proportion_function_smooth} + \eqref{eq:prop_sdna_sampling}}{\geq}}  \lambda_{\min}\left( \E\left[\bM_{[S]}^{-1}\right]\right).
\end{equation}
Observe that both the selection rule and the above bound is a generalization of the greedy coordinate sampling to minibatches. Also observe that the above bound is potentially very loose. As an example, consider a diagonal matrix $\bM$ with all the elements equal to 1. The right-hand side of \eqref{eq:prop_minibatch_greedy} is then equal to $\tau/n$, while the left-hand side is equal to $1$. Even in this very special case, the bound is disregarding a factor of $n/\tau$, which is potentially huge. For this reason, it is expected that the greedy minibatches will actually perform much better in practice than in theory \eqref{eq:prop_minibatch_greedy}.
\end{itemize}
\end{itemize}

\subsection{Non-smooth problems} \label{sec:proportion_nonsmooth}

Let us define for all $i$ and all $\bx \in \R^n$ the function \begin{equation} \label{eq:proximal_Di}
\lambda_i(\bx) \qq{\eqdef} - L \cdot \min_{v \in \R} \left\{\nabla_i f(\bx) v + \frac{L}{2}v^2 + g_i(x_i + v) - g_i(x_i)\right\}.
\end{equation}
Using the definition of proportion function \eqref{eq:proportion_function} with the diagonal $\bM \eqdef L \mathbf{I}$ we get \begin{equation} \label{eq:proportion_non-smooth}
\theta(S, \bx) \qq{\stackrel{\eqref{eq:proportion_function} + \eqref{eq:proximal_Di}}{=}} \frac{\sum_{i \in S} \lambda_i(\bx)}{L \sum_{j=1}^n \lambda_j(\bx)}.
\end{equation}  

As mentioned in Section~\ref{sec:smoothness}, the value $L$ can be chosen as $L = \lambda_{\max}(\bM)$ to satisfy the inequality \eqref{eq:Msmoothness}. However, this choice of $L$ might be suboptimal in many cases. As an example, when analyzing coordinate descent methods, the vector $\bh$ will always be of the form $h \be^i$, where $\be^i$ is the $i$-th coordinate vector (see Appendix~\ref{sec:proof_onestepbound}). Therefore, it would be sufficient for $L$ to satisfy \eqref{eq:Msmoothness} for the specific choice $\bh = h \be^i$ which leads to \[f(\bx + h\be^i) \qq{\leq} f(\bx) + \nabla_i f(\bx) h + \frac{M_{ii}}{2}h^2 \qq{\leq} f(\bx) + \nabla_i f(\bx) h  + \frac{L}{2}h^2\]
for each $i \in \{1, \dots, n\}$ and $h \in \R$. We can easily observe from the above that this will be satisfied with $L := \max_i \{M_{ii}\} \leq \lambda_{\max}(\bM)$. To account for this in general, we need to take some additional measures. Specifically, let $\cS$ be the collection of all sets $S \subset [n]$ which can be possibly generated during the iterative process by $\cP_k$. As an example, $\cS = \{\{1\}, \{2\}, \dots, \{n\}\}$ corresponds to all such sets for coordinate descent. We define $L$ as the smallest number for which \[f(\bx + \bh_{[S]}) \qq{\leq} f(\bx) + \langle \nabla f(\bx), \bh_{[S]}\rangle + \frac{1}{2}\langle \bM \bh_{[S]}, \bh_{[S]} \rangle\]
is satisfied for every $\bx, \bh \in \R^n$ and every $S \in \cS$.  It is straightforward to see that if the size of possible sets $S \in \cS$ is upper bounded as  $|S| \leq \tau$, then we can safely choose \begin{equation} \label{eq:L_general}
L \qq{=} L_\tau \qq{\eqdef} \max_{S : |S| = \tau} \{\lambda_{\max}(\bM_S)\} \qq{\leq} \max_{S:|S|=\tau}\left\{\sum_{i \in S} M_{ii} \right\},
\end{equation}
where the last inequality is due to the eigenvalues being positive and their sum being the trace.
We can easily verify that this generalizes to $L = L_{n} = \lambda_{\max}(\bM)$ for gradient descent and $L = L_1 = \max_i \{M_{ii}\}$ for coordinate descent.

Now, we will breakdown the cases depending on the block selection procedure.
\begin{itemize}
\item \textbf{Proximal Gradient Descent:} The first result in the proximal setting was the Iterative Shrinkage Tresholding Algorithm (ISTA), which selects all the coordinates on each iteration. The bound on the proportion function is given by \begin{equation} \label{eq:proportion_ISTA}
\theta([n], \bx) \qq{\stackrel{\eqref{eq:proportion_non-smooth}}{=}} \frac{\sum_{i = 1}^n \lambda_i(\bx)}{L_n  \sum_{j=1}^n \lambda_j(\bx)} \qq{=} \frac{1}{L_n} \qq{\stackrel{\eqref{eq:L_general}}{=}} \frac{1}{\lambda_{\max}(\bM)}. 
\end{equation}
\item \textbf{Serial Proximal Coordinate descent:} Suppose $|S| = 1$ and specifically let $S = \{i\}.$ Then we have that \begin{equation} \label{eq:proportion_non-smooth_serial}
\theta(\{i\}, \bx) \qq{\stackrel{\eqref{eq:proportion_non-smooth}}{=}} \frac{\lambda_i(\bx)}{L_1 \sum_{j=1}^n \lambda_j(\bx)},
\end{equation}
as the size of the selected sets are upper bounded by $1$. The procedure leading to the choice of $i$ distinguishes between various serial approaches. 
\begin{itemize}
\item \textbf{Uniform probabilities:} Assume we choose the coordinate $i$ uniformly at random at each iteration. Then we can bound the expectation of the proportion function as 
\begin{eqnarray} 
\E[\theta(\{i\}, \bx) ~|~ \bx] & \qq{\stackrel{\eqref{eq:proportion_non-smooth_serial}}{=}} & \E \left[ \frac{\lambda_i(\bx)}{L_1 \sum_{j=1}^n \lambda_j(\bx)} \right] \qq{=} \frac{\frac{1}{n}\sum_{i=1}^n \lambda_i(\bx)}{L_1 \sum_{j=1}^n \lambda_j(\bx)} \qq{=} \frac{1}{nL_1} \nonumber \\
& \qq{\stackrel{\eqref{eq:L_general}}{=}} & \frac{1}{n \max_i \{M_{ii}\}}. \label{eq:proportion_non-smooth_serial_uniform}
\end{eqnarray}
\item \textbf{Greedy choice:} Another approach is to pick the coordinate $i$, which maximizes the proportion function in \eqref{eq:proportion_non-smooth_serial}, which was analyzed in \cite{nutini2015coordinate}. The best bound independent on $\bx$ coincides with the above bound for uniform sampling, using the fact that a mean of some quantity is less than its maximum
\begin{eqnarray} 
\theta(\{i\}, \bx) & \qq{\stackrel{\eqref{eq:proportion_non-smooth_serial}}{=}} & \frac{\max_i \{\lambda_i(\bx)\}}{L_1\sum_{j=1}^n \lambda_j(\bx)} \qq{\geq}\frac{\frac{1}{n}\sum_{i=1}^n \lambda_i(\bx)}{L_1\sum_{j=1}^n \lambda_j(\bx)} \qq{=} \frac{1}{nL_1} \nonumber \\ & \qq{\stackrel{\eqref{eq:L_general}}{=}} & \frac{1}{n \max_{i}\{M_{ii}\}}. \label{eq:proportion_non-smooth_serial_greedy}
\end{eqnarray}
Similarly as in the smooth case, observe that the quantity $\max_i\{\lambda_i(\bx)\}$ is potentially up to $n$ times larger than $\frac{1}{n}\sum_{i=1}^n \lambda_i(\bx)$ and therefore the bound \eqref{eq:proportion_non-smooth_serial_greedy} can be potentially $n$ times larger in some cases, which usually results in better empirical results.
\end{itemize}
\item \textbf{Minibatch Proximal Coordinate Descent:} Suppose $|S| = \tau$, where $\tau$ is given, which implies that we can use $L = L_{\tau}$ in the bounds. We introduce two options for the block selection procedure.
\begin{itemize}
\item $\tau$\textbf{-nice sampling:} Only one variant of a sampling for this setup was considered before \cite{PCDM}, and that is a uniform choice of $\tau$ coordinates without repetition.  In expectation, each coordinate has a chance of $\tau/n$ to be picked, which is used in the bound to get  
\begin{eqnarray}
\E[\theta(S, \bx) ~|~ \bx] & \qq{\stackrel{\eqref{eq:proportion_non-smooth}}{=}} & \E \left[ \frac{\sum_{i \in S} \lambda_i(\bx)}{L_\tau\sum_{j=1}^n \lambda_j(\bx)} \right] \qq{=} \frac{\frac{\tau}{n}\sum_{i=1}^n \lambda_i(\bx)}{L_\tau\sum_{j=1}^n \lambda_j(\bx)} \qq{=} \frac{\tau}{nL_\tau} \nonumber \\ & \qq{\stackrel{\eqref{eq:L_general}}{=}} & \frac{\tau}{n \max_{S : |S| = \tau}\{\lambda_{\max}(\bM_{S})\}}. \label{eq:proportion_non-smooth_minibatch}
\end{eqnarray}
To our best knowledge, this bound is new, as the previous bound considered $L_n$ instead of $L_{\tau}$. As $L_{\tau} \leq L_n$, the new bound is better.
\item \textbf{Greedy minibatches:} Similarly as in the smooth case, we introduce a new selection rule -- greedy minibatches. Specifically, the corresponding set $S^g$ is given by \begin{equation} \label{eq:greedy_rule_minibatch_nonsmooth}
S^g \qq{\eqdef} \argmax_{S : |S| = \tau} \left\{ \sum_{i \in S} \lambda_i(\bx) \right\}.
\end{equation}
Note that for $\tau=1$ we recover the greedy coordinate descent. To give global bounds independent of $\bx$ for this strategy, we again use the fact that maximum is an upper bound for the mean, to get \begin{eqnarray}
\theta(S^g, \bx) & \qq{\stackrel{\eqref{eq:proportion_non-smooth}}{=}} &  \frac{\sum_{i \in S^g} \lambda_i(\bx)}{L_\tau\sum_{j=1}^n \lambda_j(\bx)} \qq{\geq} \frac{\frac{\tau}{n}\sum_{i=1}^n \lambda_i(\bx)}{L_\tau\sum_{j=1}^n \lambda_j(\bx)} \qq{=} \frac{\tau}{nL_\tau} \nonumber \\ & \qq{\stackrel{\eqref{eq:L_general}}{=}} & \frac{\tau}{n \max_{S : |S| = \tau}\{\lambda_{\max}(\bM_{S})\}}. \label{eq:proportion_non-smooth_minibatch_greedy}
\end{eqnarray}
Note that the the above bound can be potentially very pessimistic, as \[ 1 \qq{\geq} \frac{\sum_{i=1}^n \lambda_i(\bx)}{\sum_{i \in S^g} \lambda_i(\bx)} \qq{\geq} \frac{\tau}{n}\]
by using $\sum_{i \in S^g} \lambda_i(\bx) \geq \frac{\tau}{n}\sum_{i=1}^n \lambda_i(\bx) \geq \frac{\tau}{n}\sum_{i \in S^g} \lambda_i(\bx) $. Therefore, the bound \eqref{eq:proportion_non-smooth_minibatch_greedy} can be up to $n/\tau$ times better in certain cases.
\end{itemize}
\end{itemize}

\section{Strongly Polyak-\L{}ojasiewicz Functions} \label{sec:SPL}

In this section, we reinvent the strongly PL functions using the proximal forcing function \eqref{eq:GPL_function}, and develop the corresponding convergence rates. Also, we show how to recover the known results in this setting by applying our theory.

\subsection{Strongly PL functions}
\begin{definition}[Strongly PL functions: composite case] \label{def:strongPL-composite} We say that $F$ is a strongly PL function there exists a scalar $\mu > 0$ satisfying
\begin{equation} \label{eq:s98g98f898d} 
\mu(\bx) \qq{\geq} \mu 
\end{equation} for all $\bx \in {\rm dom}(g) / \cX^*$. The collection of all functions $F$ satisfying inequality \eqref{eq:s98g98f898d} will be denoted $\cS_{PL}^g(\mu)$, and we say that $F$ is strongly PL with parameter $\mu$.
\end{definition}

Recall that in the smooth case we said that a function $f \in \cS_{PL}(\mu)$, if $f$ satisfied the condition \eqref{eq:s09h9hf3}, which can be observed to be equivalent to \eqref{eq:s98g98f898d} for smooth functions. Therefore we have that $\cS_{PL}(\mu) \subset \cS^g_{PL}(\mu)$.

The above definition is not new, it was originally introduced in a slightly different form by Karimi et al.~\cite{karimi2016linear}.

\subsection{Strongly convex functions are strongly PL} \label{sec:SPLg}

Let $\lambda \geq 0$. Function $F:\R^n \rightarrow \extR$ is said to be \textit{$\lambda$-strongly convex}, if for all $\bx, \by \in \R^n$ and all $\beta \in [0,1]$ we have 
\begin{equation} \label{eq:strongly_convex}
F(\beta \bx + (1-\beta)\by) \qq{\leq} \beta F(\bx) + (1-\beta) F(\by) - \frac{\lambda\beta(1- \beta)}{2}\|\bx - \by\|^2.
\end{equation}
If $F$ is $0$-strongly convex, we refer to it simply as \textit{convex.} Consider a differentiable function $f:\R^n \rightarrow \R$. If $f$ is $\lambda$-strongly convex for $\lambda\geq 0$, then
\begin{equation} \label{eq:strongly_convex_smooth}
f(\bx + \bh) \qq{\geq} f(\bx) + \langle \nabla f(\bx), \bh \rangle + \frac{\lambda}{2}\|\bh\|^2, \qquad \bx,\bh \in \R^n.
\end{equation} 

We will now show that if $F$ is strongly convex, then $F\in \cS^g_{PL}(\rho)$ for some specific   $\rho$ and $\cA$. This means that the class of strongly PL (composite) functions contains the class of strongly convex (composite) functions. 

\begin{theorem} \label{thm:GPL_convex_strongly_convexXXX}
Assume $F$ is $\lambda_F$-strongly convex with $\lambda_F>0$, and $f$ is  $\lambda_f$-strongly convex with $\lambda_f\geq 0$.  Then \begin{equation} \label{eq:GPL_function_strongly_convex}
\mu(\bx) \qq{\geq}  \mu \qq{\eqdef} \min \left\{ \frac{L}{2}, 
\frac{L \lambda_F}{\lambda_F - \lambda_f + L} 
 \right\}
\end{equation}
for all $\bx\in {\rm dom}(g)/\cX^*$, and hence $F\in \cS^g_{PL}(\mu)$.
\end{theorem}
\begin{proof}
See Section~\ref{sec:proof_GPL_strongly_convex}.
\end{proof}

In the above theorem we do not enforce separability assumption on $g$.


\subsection{Convergence}\label{sec:8h9d8h8fsd}

We have the  following convergence result, establishing convergence to a global minimizer.

\begin{theorem} \label{thm:SPLg}
Invoke Assumption~\ref{ass:non-convex_setup}. Further, assume $F \in \cS_{PL}^g(\mu)$ (i.e., $F$ is strongly PL with parameter $\mu$), and let 
\begin{equation} \label{eq:SPLg_muk}
\mu_{k} \qq{\eqdef} \frac{\mu \E[\xi(\bx^k) \cdot \theta (S_k, \bx^k)]}{\E[\xi(\bx^k)]}.
\end{equation}
Then \begin{equation} \label{eq:SPLg_convergence_rate} 
\sum_{k=0}^{K-1}\mu_k \geq \log \left(\frac{\xi(\bx^0)}{\epsilon}\right) \quad \Rightarrow \quad \E[\xi(\bx^K)] \leq \epsilon.
\end{equation}
\end{theorem}
\begin{proof}
See Section~\ref{sec:proof_SPLg_general}.
\end{proof}

In order to get concrete complexity results from the above theorem, we need to estimate the speed of growth of $\sum_{k=0}^{K-1} \mu_k$ in $K$. There no universal way to do this, which is why we state the above result the way we do. Instead, in each situation this needs to be estimated separately. Typically, this will be done by lower bounding $\mu_k$ for each $k$ separately. Let us illustrate this using a couple examples.

If the blocks are selected {\em deterministically}, then the expectations in \eqref{eq:SPLg_muk} do not play any role, and we have $\mu_k = \mu \theta(S_k, \bx^k)$. If we have a global lower bound of the form \[\theta(S_k, \bx^k) \qq{\geq} c \qq{>} 0\] readily available, then $\sum_{k=0}^{K-1} \mu_k \geq \mu c K$, which implies the rate  
 \begin{equation} \label{eq:SPLg_convergence_rateXX} 
K \geq \frac{1}{c\mu} \log \left(\frac{\xi(\bx^0)}{\epsilon}\right) \quad \Rightarrow \quad \xi(\bx^K) \leq \epsilon
\end{equation}
or, equivalently, $\xi(\bx^K) \leq  \xi(\bx^0) \cdot e^{-c \mu K}$. More generally, convergence is established whenever we can lower bound  $\theta(S_k, \bx^k) \geq c_k >0$, where the constants $c_k$ sum up to infinity.

If the blocks are selected {\em stochastically}, then the sequence of iterates $\bx^k$ is also stochastic. In such cases, it is often possible to come up with a bound  for the expectation of $\theta(S_k, \bx^k)$ conditioned on $\bx^k$:
\begin{equation} \label{eq:independent_SPLg}
\E[\theta(S_k, \bx^k) ~|~ \bx^k ] \qq{\geq} c \qq{>} 0.
\end{equation}
If this is the case, we claim that $\mu_k$ can be lower bounded by $\mu c$, and hence Theorem~\ref{thm:SPLg} implies the same rate as before: given by \eqref{eq:SPLg_convergence_rateXX}, only with a bound on the expectation $\E[\xi(\bx^k)] \leq \epsilon$ on the right-hand side. More generally, if $\E[\theta(S_k, \bx^k) ~|~ \bx^k ] \geq c_k \geq 0$, then $\mu_k 
\geq \mu c_k$, and convergence is guaranteed as long as $\sum_{k}c_k = \infty$.

Let us now return to the claim. Indeed,  the bound $\mu_k \geq \mu c$ follows from \eqref{eq:independent_SPLg} by applying Lemma~\ref{lem:onestep_expectation} (see Appendix) with $X := \bx^k, Y := \theta_k$ and  $f(\bx) := \xi(\bx)$.


\subsection{Applications of Theorem~\ref{thm:SPLg}}

We now showcase the use of Theorem~\ref{thm:SPLg} on  selected algorithms which arise as special cases of our generic method (Algorithm~\ref{eq:US_min_nonsmooth}).

\paragraph{Proximal gradient descent.}
If for all $k$ we choose $S_k=[n]$ with probability 1,  
Algorithm~\ref{eq:US_min_nonsmooth} reduces to (proximal) gradient descent. In both smooth and non-smooth case we have $\theta([n],\bx^k) \geq 1/\lambda_{\max}(\bM) $ (see \eqref{eq:prop_lower_bound_smooth_batch} and \eqref{eq:proportion_ISTA}). Substituting into \eqref{eq:SPLg_convergence_rateXX}, we get the rate
\[K  \geq \frac{L}{\mu}\log \left(\frac{\xi(\bx^0)}{\epsilon}\right) \quad \Rightarrow \quad \xi(\bx^K) \leq \epsilon,\]
in both the smooth and non-smooth case. These results were previously  obtained for strongly PL functions in \cite{karimi2016linear}.

\paragraph{Randomized coordinate descent with uniform probabilities.}

Randomized coordinate descent, analyzed in \cite{Nesterov:2010RCDM, UCDC}, arises as special case of 
Algorithm~\ref{eq:US_min_nonsmooth} by choosing $S_k = \{i_k\}$, where $i_k$ is an index chosen from $[n]$ uniformly at random, and independently of the history of the method.

Note that a bound on $\E[\theta(S_k, \bx^k ~|~ \bx^k)]$ is readily available in Section~\ref{sec:proportion}, specifically in \eqref{eq:prop_lower_bound_rcd_uniform} for the smooth case and \eqref{eq:proportion_non-smooth_serial_uniform} in the non-smooth case, and it takes the form \[\E[\theta(S_k, \bx^k)  ~|~ \bx^k] \qq{\geq} c \qq{\eqdef} \tfrac{1}{ n \max_i \{M_{ii}\} }. \] As we have seen in the discussion immediately following Theorem~\ref{thm:SPLg},   this implies the bound $\mu_k \geq \mu c$. Applying Theorem~\ref{thm:SPLg},  we conclude that \[
K\geq  \frac{n \max\{M_{ii}\}}{\mu}\log\left(\frac{\xi(\bx^0)}{\epsilon}\right) \quad \Rightarrow \quad \E[\xi(\bx^K)] \leq \epsilon,
\]
which is the same bound as given in \cite{karimi2016linear} for strongly PL functions in both smooth and non-smooth case. 

\paragraph{Randomized coordinate descent with importance sampling}

We now allow for specific nonuniform probabilities: probability of choosing $S_k=\{i\}$ is proportional to $M_{ii}$ (see \eqref{eq:imp_sampling}). In view of  \eqref{eq:prop_lower_bound_rcd_importance}, we get $\E[\theta(S_k, \bx^k ~|~ \bx^k)] \geq c \eqdef 1/\sum_{i=1}^n M_{ii}$ for smooth functions. This leads to the complexity result
\[K  \geq \frac{\sum_{i=1}^nM_{ii}}{\mu}\log \left(\frac{\xi(\bx^0)}{\epsilon}\right) \quad \Rightarrow \quad \E[\xi(\bx^K)] \leq \epsilon,\]
which is {\em is a new result} for strongly PL functions. However, in  the special case of strongly-convex functions, this result is known \cite{Nesterov:2010RCDM, UCDC,NSync}.

\paragraph{Minibatch coordinate descent.}
Assume $S_k$ is a subset of $[n]$ of cardinality $\tau$, chosen uniformly at random. This leads to (strandard) minibatch coordinate descent. In view of \eqref{eq:prop_sdna_sampling}, we have the bound $\E[\theta(S, \bx) ~|~ \bx] \geq  \lambda_{\min}\left( \E\left[\bM_{[S]}^{-1}\right]\right)$ in the smooth case. Substituting into \eqref{eq:SPLg_convergence_rateXX}, we get the rate
\[K  \geq \frac{1}{\mu \lambda_{\min}(\E[\bM_{S}^{-1}])}\log \left(\frac{\xi(\bx^0)}{\epsilon}\right) \quad \Rightarrow \quad \E[\xi(\bx^K)] \leq \epsilon.\]
This result is novel for the strongly PL case, but it was already established before for strongly convex functions in \cite{SDNA}. 

In the non-smooth case we can use the bound \eqref{eq:proportion_non-smooth_minibatch} given by $\E[\theta(S, \bx) ~|~ \bx] \geq \tau/nL_{\tau}$, where $L_\tau$ is given by \eqref{eq:L_general}. It follows that we have the guarantee \[K  \geq \frac{n L_{\tau}}{\mu \tau}\log \left(\frac{\xi(\bx^0)}{\epsilon}\right) \quad \Rightarrow \quad \E[\xi(\bx^K)] \leq \epsilon,\]
which is novel for strongly PL and also for strongly convex functions, as $L_\tau$ is less or equal to the standard $L_n$ \cite{PCDM}.

\paragraph{Greedy coordinate descent.} Assume a single coordinate $S = \{i\}$ is chosen using the rule stated in \eqref{eq:greedy_rule_smooth}, i.e., the choice maximizes the proportion function on the given iteration. In the smooth case we have the bound  
\eqref{eq:prop_lower_bound_rcd_greedy}, i.e., $\theta(S_k,\bx^k) \geq c \eqdef 1/\sum_{i=1}^n M_{ii}$ . This leads to the rate
\[K \geq \frac{\sum_{i=1}^nM_{ii}}{\mu}  \log \left(\frac{\xi(\bx^0)}{\epsilon}\right) \quad \Rightarrow \quad \xi(\bx^K) \leq \epsilon.\]
Note that this is identical to the rate of randomized coordinate descent with importance sampling, with the exception that we have $\xi(\bx^K)$ instead of $\E[\xi(\bx^K)]$ due to the deterministic nature of the method.
Again, this result was already established in \cite{karimi2016linear}.  In the special case of strongly convex functions, this was first established by Nesterov~\cite{Nesterov:2010RCDM}.

As for the proximal case, we have a bound \eqref{eq:proportion_non-smooth_serial_greedy} given by $\theta(S, \bx^k) \geq c \eqdef 1/n \max_i\{M_{ii}\}$ which leads to a similar rate
\[K \geq \frac{n \max_{i}\{M_{ii}\}}{\mu}  \log \left(\frac{\xi(\bx^0)}{\epsilon}\right) \quad \Rightarrow \quad \xi(\bx^K) \leq \epsilon.\]
This rate is identical to the rate of randomized coordinate descent, except for the expectation on the right-hand side. This result was also estabilished in \cite{karimi2016linear}, while the strongly convex version is due \cite{nutini2015coordinate}.

\paragraph{Greedy minibatches.}

Assume the smooth and and that we choose the set of coordinates $S$ according to the rule described in \eqref{eq:greedy_rule_minibatch_smooth}, which is \[S = S(\bx) \qq{\eqdef} \argmax_{S \;:\; |S| = \tau} \left\{ (\nabla_S f(\bx))^\top (\bM_{S})^{-1} \nabla_S f (\bx) \right\}.\] Assuming that $|S| = \tau$, we can see that this rule maximizes the proportion function for the given iteration. The corresponding lower bound takes the form \eqref{eq:prop_minibatch_greedy} which leads to the rate

\[K \geq \frac{1}{\lambda_{\min}( \E[\bM_{[S]}^{-1}])\mu}  \log \left(\frac{\xi(\bx^0)}{\epsilon}\right) \quad \Rightarrow \quad \xi(\bx^K) \leq \epsilon.\]
Similarly as in the serial case, this bound is identical to uniform minibatches, with the only exception being the dropped expectation on the optimality gap $\xi(\bx)$, due to this algorithm being deterministic.

In the non-smooth case we have the bound on the proportion function given by \eqref{eq:proportion_non-smooth_minibatch_greedy} which leads to the rate
\[K \geq \frac{n L_{\tau}}{\mu \tau}  \log \left(\frac{\xi(\bx^0)}{\epsilon}\right) \quad \Rightarrow \quad \xi(\bx^K) \leq \epsilon.\]
Both of the above bounds are novel, as greedy minibatches is a novel sampling approach.

\section{Weakly Polyak-\L{}ojasiewicz Functions} \label{sec:WPLg}

In this section we introduce a generalized definition of Weakly PL functions, which specifies to Definition~\ref{def:weaklyPL_smooth} for a specific choice of parameters, and also covers the proximal case. We show that proximal convex functions can be analyzed using this framework and we give a general convergence rate guarantee for this class. Lastly, we specify our theory to several known setups, showcasing the generality of our definition.

\subsection{Weakly PL functions}

\begin{definition}[Weakly PL functions: general case] \label{def:weaklyPL-general} 
We say that $F$ is a weakly PL function, if there exists a scalar function $\rho: \R^n \rightarrow \R^+$, such that
\begin{equation} \label{eq:weakly_PL_general} \mu(\bx) \qq{\geq} \rho(\bx^0) \cdot \xi(\bx) \qq{>} 0,
\end{equation} for all $\bx^0, \bx \in {\rm dom}(g) / \cX^*$ such that $\xi(\bx) \leq \xi(\bx^0)$. The collection of all functions $F$ satisfying inequality \eqref{eq:weakly_PL_general} will be denoted $\cW_{PL}^{g}(\rho)$, and we say that $F$ is weakly PL with parameter $\rho$.
\end{definition}

Recall that in the smooth case we said that a function $f \in \cW_{PL}(\mu)$, if $f$ satisfied the condition \eqref{eq:nd09un09s}
\[\mu(\bx) \geq  \frac{\mu \xi(\bx)}{2\|\bx-\bx^*\|^2}, \qquad \bx\in \R^n/\cX^*. \]
In the proof of convergence of these methods, we further bounded the right-hand side of the above expression by $\mu \xi(\bx)/ 2R^2$, where we defined $R\eqdef  \max_{\bx\in \cL_0} \|\bx-\bx^*\| < +\infty$, with $\cL_0\eqdef \{\bx\in \R^n \;:\; f(\bx) \leq f(\bx^0)\}$. To see the above smooth case in our general framework defined in \eqref{eq:weakly_PL_general}, we can use $\rho(\bx^0) = \frac{1}{2R^2},$ which satisfies our assumption, as $R$ does not depend on $\bx$.

\subsection{Weakly convex functions are weakly PL}

We will now show that if $F$ is convex, then $F\in \cW^{g}_{PL}(\rho)$ for some specific $\rho$. This means that weakly PL functions generalize weakly convex functions also in the composite setting.

\begin{theorem} \label{thm:GPL_convex_weakly_convexXXX}
Assume $f$ and $g$ are convex, and let $\bx^*$ be a global minimizer of $F$.  Then \begin{equation} \label{eq:GPL_function_convex}
\mu(\bx) \qq{\geq}  \xi(\bx) \cdot \min \left\{ \frac{1}{2 \xi(\bx)},\frac{1 }{2L\|\bx - \bx^*\|^2} \right\} 
\end{equation}
for all $\bx\in {\rm dom}(g)/\cX^*$. Also, $F$ is a weakly PL function with the parameter $\rho$ given by
\begin{equation} \label{eq:weakly_PL_rho_convergence}
\rho(\bx^0) \qq{=} \min \left\{\frac{L}{2 \xi(\bx^0)}, \frac{1}{2 R^2} \right\},
\end{equation}
where \begin{equation} \label{eq:weakly_PL_R_convergence}
R \qq{\eqdef}  \max_{\bx \in \R^n : f(\bx) \leq f(\bx^0)} \|\bx-\bx^*\| \qq{<} +\infty.
\end{equation}
\end{theorem}
\begin{proof}
See Section~\ref{sec:proof_GPL_convex}.
\end{proof}
Note that in the smooth case we have $\xi(\bx) \leq \frac{L}{2}\|\bx - \bx^*\|^2$, therefore 
\[\min \left\{ \frac{1}{2 \xi(\bx)},\frac{1 }{2L\|\bx - \bx^*\|^2} \right\} \qq{=} \frac{1 }{2L\|\bx - \bx^*\|^2},\]
which also leads to $\rho(\bx^0) = 1/(2R^2)$ in the smooth case.
\subsection{Convergence}

For this class, we can show the following convergence result.
\begin{theorem} \label{thm:WPLg}
Invoke Assumption~\ref{ass:non-convex_setup}. Further, assume $F \in \cW_{PL}^{g}(\rho)$ (i.e., $F$ is weakly PL with parameter $\rho$). Let \begin{equation} \label{eq:WPLg_muk}
\mu_{k} \qq{\eqdef} \rho(\bx^0)\frac{\E[(\xi(\bx^k))^2 \cdot \theta(S_k, \bx^k)]}{(\E[\xi(\bx^k)])^2}.
\end{equation}
Then \begin{equation} \label{eq:WPLg_convergence_rate} 
\sum_{k=0}^{K-1}\mu_k \geq \frac{1}{\epsilon} \quad \Rightarrow \quad \E[\xi(\bx^K)] \leq \epsilon.
\end{equation}

\end{theorem}
\begin{proof}
See Section~\ref{sec:proof_WPLg_general}.
\end{proof}

Similarly as in the previous section, we need to bound the quantity $\sum_{k=0}^K \mu_k$ in $K$ to get a complexity result. The standard theory is developed by bounding each of $\mu_k$ separately, although this is apparently not the optimal way.

In the case that the blocks are selected deterministically, then the expectations in \eqref{eq:WPLg_muk} do not play any role and we have $\mu_k = \rho(\bx^0) \theta(S_k, \bx^k)$. Additionally, if we have a global lower bound $\theta(S_k, \bx^k) \geq c > 0$, then $\sum_{k=0}^{K-1}\mu_k \geq \rho(\bx^0) c K$, which implies the rate \[K \geq \frac{1}{\rho(\bx^0) c \epsilon} \qq{\Rightarrow} \xi(\bx^K) \leq \epsilon\]
and equivalently $\xi(\bx^K) \leq \frac{1}{\rho(\bx^0) c K}$.

We get a similar result for stochastic block selection. Specifically, if we have a global bound for the quantity $\E[\theta(S_k, \bx^k) ~|~ \bx^k] \geq c > 0,$ we can lower bound the values of $\mu_k$ in Theorem~\ref{thm:WPLg} by $\rho(
\bx^0) c$. This can be done by first lower bounding the denominator of \eqref{eq:WPLg_muk} by $\E[(\xi(\bx^k))^2]$ using the trivial bound $\E[X]^2 \leq \E[X^2]$ and further applying Lemma~\ref{lem:onestep_expectation} with $X := \bx^k, Y := \theta_k, f(\bx) = (\xi(\bx))^2$.

In general, we can claim convergence if we have a sequence of lower bounds $\{c_k\}_{k=0}^\infty$ for $\theta(S_k, \bx^k) \geq c_k$ in the deterministic case or $\E[\theta(S_k, \bx^k) ~|~ \bx^k] \geq c_k$ in the stochastic case, and additionally $\sum_{k=0}^\infty c_k = \infty$.

\subsection{Applications of Theorem~\ref{thm:WPLg}}

We now showcase the use of Theorem~\ref{thm:WPLg} on several methods which arise as special cases of our generic method (Algorithm~\ref{eq:US_min_nonsmooth}). All the results for general $\cW_{PL}^g$ functions are novel, as the notion itself is novel. In most cases, we recover known theory by specializing the results to convex objectives.

\paragraph{Proximal gradient descent.}
In both smooth and non-smooth case we have $\theta([n],\bx^k) \geq c \eqdef 1/\lambda_{\max}(\bM) $ (see \eqref{eq:prop_lower_bound_smooth_batch} and \eqref{eq:proportion_ISTA}). Substituting into \eqref{eq:WPLg_convergence_rate}, we get the rate
\[K  \geq \frac{\lambda_{\max}(\bM)}{\rho(\bx^0)\epsilon} \quad \Rightarrow \quad \xi(\bx^K) \leq \epsilon,\]
in both the smooth and non-smooth case. This result is novel for Weakly PL functions, but it was estabilished before for convex functions in \cite{nesterov2013introductory}.

\paragraph{Randomized coordinate descent with uniform probabilities.}

The bound on the quantity $\E[\theta(S_k, \bx^k ~|~ \bx^k)]$ is readily available in Section~\ref{sec:proportion}, specifically in \eqref{eq:prop_lower_bound_rcd_uniform} for the smooth case and \eqref{eq:proportion_non-smooth_serial_uniform} in the non-smooth case, and it takes the form $\E[\theta(S_k, \bx^k)  ~|~ \bx^k] \geq c  \eqdef \tfrac{1}{ n \max_i \{M_{ii}\} }.$ As discussed right after Theorem~\ref{thm:WPLg},   this implies the bound $\mu_k \geq \mu c$. Applying Theorem~\ref{thm:WPLg},  we conclude that \[
K\geq  \frac{n \max\{M_{ii}\}}{\rho(\bx^0)\epsilon} \quad \Rightarrow \quad \E[\xi(\bx^K)] \leq \epsilon,
\]
which is novel for weakly PL functions, but it is well known for convex objectives \cite{Nesterov:2010RCDM}.

\paragraph{Randomized coordinate descent with importance sampling}

In view of  \eqref{eq:prop_lower_bound_rcd_importance}, we get the lower bound $\E[\theta(S_k, \bx^k ~|~ \bx^k)] \geq c \eqdef 1/\sum_{i=1}^n M_{ii}$ for smooth functions. This leads to the complexity result
\[K  \geq \frac{\sum_{i=1}^nM_{ii}}{\rho(\bx^0) \epsilon} \quad \Rightarrow \quad \E[\xi(\bx^K)] \leq \epsilon,\]
which is a novel result for weakly PL functions. However, in the special case of convex functions, this result is known \cite{Nesterov:2010RCDM}.

\paragraph{Minibatch coordinate descent.}
 In view of \eqref{eq:prop_sdna_sampling}, we have the lower bound $\E[\theta(S, \bx) ~|~ \bx] \geq  \lambda_{\min}\left( \E\left[\bM_{[S]}^{-1}\right]\right)$ in the smooth case. Substituting into \eqref{eq:WPLg_convergence_rate}, we get the rate
\[K  \geq \frac{1}{\rho(\bx^0) \lambda_{\min}(\E[\bM_{S}^{-1}]) \epsilon} \quad \Rightarrow \quad \E[\xi(\bx^K)] \leq \epsilon.\]
This result is novel for the weakly PL case and to the best of our knowledge also for the convex case, as the results in \cite{SDNA} only consider strongly convex objectives.

In the non-smooth case we can use the bound \eqref{eq:proportion_non-smooth_minibatch} given by $\E[\theta(S, \bx) ~|~ \bx] \geq \tau/nL_{\tau}$, where $L_\tau$ is given by \eqref{eq:L_general}. It follows that we have the guarantee \[K  \geq \frac{n L_{\tau}}{\rho(\bx^0) \tau \epsilon} \quad \Rightarrow \quad \E[\xi(\bx^K)] \leq \epsilon,\]
which is novel for weakly PL and to the best of our knowledge also for weakly convex functions, as $L_\tau$ is less or equal to the standard $L_n$ \cite{PCDM}.

\paragraph{Greedy coordinate descent.} In the smooth case we have the bound  
\eqref{eq:prop_lower_bound_rcd_greedy}, i.e., $\theta(S_k,\bx^k) \geq c \eqdef 1/\sum_{i=1}^n M_{ii}$ . This leads to the rate
\[K \geq \frac{\sum_{i=1}^nM_{ii}}{\rho(\bx^0) \epsilon} \quad \Rightarrow \quad \xi(\bx^K) \leq \epsilon.\]
This is identical to the above rate of randomized coordinate descent with importance sampling, with a dropped expectation on $\xi(\bx^0)$. This result is new for weakly PL functions and in the special case of weakly convex functions, it was first established by Nesterov~\cite{Nesterov:2010RCDM}.

In the proximal case, we have a bound \eqref{eq:proportion_non-smooth_serial_greedy} given by $\theta(S, \bx^k) \geq c \eqdef 1/n \max_i\{M_{ii}\}$ which leads to the rate
\[K \geq \frac{n \max_{i}\{M_{ii}\}}{\rho(\bx^0) \epsilon}  \quad \Rightarrow \quad \xi(\bx^K) \leq \epsilon.\]
This rate is identical to the rate of randomized coordinate descent, except for the expectation on the right-hand side. This result is novel for weakly PL and to the best of our knowledge it is also novel for the special case of convex functions.

\paragraph{Greedy minibatches.}

In the smooth case we have the bound \eqref{eq:prop_minibatch_greedy} which leads to the rate
\[K  \geq \frac{1}{\rho(\bx^0) \lambda_{\min}(\E[\bM_{S}^{-1}]) \epsilon} \quad \Rightarrow \quad \E[\xi(\bx^K)] \leq \epsilon.\]

In the non-smooth case we have the bound on the proportion function given by \eqref{eq:proportion_non-smooth_minibatch_greedy} which leads to the rate
\[K  \geq \frac{n L_{\tau}}{\rho(\bx^0) \tau \epsilon} \quad \Rightarrow \quad \E[\xi(\bx^K)] \leq \epsilon.\]
Both of the above results are shared with uniform minibatch coordinate descent, except that we are bounding the optimality gap itself instead of its expectation.

Both of the above rates are novel, as greedy minibatches constitute a novel sampling approach.

\section{General Nonconvex Functions} \label{sec:general_nonconvex}

In this section we establish a generic convergence result applicable to general nonconvex functions. This is done at the expense of losing global optimality: we will show that either $\lambda(\bx^k)$ gets small, or that $F(\bx^k)$ is close to the global minimum $F(\bx^*)$. Recall that in the smooth case ($g=0$)  we have $\lambda(\bx^k) = \tfrac{1}{2}\|\nabla f(\bx^k)\|^2$.

\begin{theorem} \label{thm:non-convex_case}
Invoke Assumption~\ref{ass:non-convex_setup}. Let $\epsilon>0$ be fixed. Further, let
\begin{equation} \label{eq:non-convex_muk}
\mu_k \qq{\eqdef} \frac{\E[\xi(\bx^k) \theta(S_k, \bx^k)]}{\E[\xi(\bx^k)]}. 
\end{equation}

If the inequality
\begin{equation} \label{eq:convergence_non-convex_condition}
\frac{\xi(\bx^0)}{\epsilon} \log \left( \frac{\xi(\bx^0)}{\epsilon} \right) \qq{\leq} \sum_{k=0}^{K-1} \mu_k,
\end{equation}
holds, then at least one of the following conclusions holds: \begin{enumerate}
\item[(i)] $\lambda(\bx^k)< \epsilon$ for at least one $k \in \{0, \dots, K-1\}$, 
\item[(ii)] $\E[\xi(\bx^K)] \leq \epsilon$.
\end{enumerate}

\end{theorem}

\begin{proof}
See Section~\ref{sec:proof_non-convex_case}.
\end{proof}

In the smooth case ($g = 0$), condition $(i)$ reduces to  $\tfrac{1}{2}\|\nabla f(\bx^k)\|^2 \leq \epsilon$ for at least one $k \in \{0, \dots, K-1\}$. Also note that the same strategies for bounding $\mu_k$ as those outlined in \ref{sec:8h9d8h8fsd} apply here. To sum it up, if we have a global bound $\theta(S_k, \bx^k) \geq c > 0$ in the deterministic case, then it follows from Theorem~\ref{thm:non-convex_case} that \[K \geq \frac{\xi(\bx^0)}{c\epsilon}\log\left( \frac{\xi(\bx^0)}{\epsilon}\right) \qq{\Rightarrow} \left( (\xi(\bx^K) \leq \epsilon) \vee (\exists k \in [K] : \lambda(\bx^k) \leq \epsilon)\right).\]
Similarly, if we have a bound $\E[\theta(S_k,\bx^k) ~|~ \bx^k] \geq c > 0$ in the stochastic case, we get the same result as above with an expectation over the optimality gap $\E[\xi(\bx^K)]$, as the optimality gap becomes a random variable. 
\subsection{Applications of Theorem~\ref{thm:non-convex_case}} \label{sec:general_non-convex_applications}

In this part we apply the results from Theorem~\ref{thm:non-convex_case} to several known methods to acquire local convergence guarantees. To the best of our knowledge, all the results in this section are novel.

\paragraph{Proximal gradient descent.}
In both smooth and non-smooth case we have $\theta([n],\bx^k) \geq c \eqdef 1/\lambda_{\max}(\bM) $ (see \eqref{eq:prop_lower_bound_smooth_batch} and \eqref{eq:proportion_ISTA}). Substituting into \eqref{eq:convergence_non-convex_condition}, we get the rate
\[K \geq \frac{\lambda_{\max}(\bM)\xi(\bx^0)}{\epsilon}\log\left( \frac{\xi(\bx^0)}{\epsilon}\right) \qq{\Rightarrow} \left( (\xi(\bx^K) \leq \epsilon) \vee (\exists k \in [K] : \lambda(\bx^k) \leq \epsilon)\right)\]
in both the smooth and non-smooth case.

\paragraph{Randomized coordinate descent with uniform probabilities.}

In both smooth and non-smooth case we have the bound $\E[\theta(S_k, \bx^k)  ~|~ \bx^k] \geq c  \eqdef \tfrac{1}{ n \max_i \{M_{ii}\} }$ (see \eqref{eq:prop_lower_bound_rcd_uniform} and \eqref{eq:proportion_non-smooth_serial_uniform}). This implies the rate \[K \geq \frac{\max_{i}\{M_{ii}\}\xi(\bx^0)}{\epsilon}\log\left( \frac{\xi(\bx^0)}{\epsilon}\right) \qq{\Rightarrow} \left( (\E[\xi(\bx^K)] \leq \epsilon) \vee (\exists k \in [K] : \lambda(\bx^k) \leq \epsilon)\right)\]
for both smooth and non-smooth functions.

\paragraph{Randomized coordinate descent with importance sampling}

In view of  \eqref{eq:prop_lower_bound_rcd_importance}, we have that $\E[\theta(S_k, \bx^k ~|~ \bx^k)] \geq c \eqdef 1/\sum_{i=1}^n M_{ii}$ for smooth functions. This leads to the complexity result
\[K \geq \frac{\xi(\bx^0)\sum_{i=1}^nM_{ii}}{\epsilon}\log\left( \frac{\xi(\bx^0)}{\epsilon}\right) \qq{\Rightarrow} \left( (\E[\xi(\bx^K)] \leq \epsilon) \vee (\exists k \in [K] : \lambda(\bx^k) \leq \epsilon)\right).\]

\paragraph{Minibatch coordinate descent.}
We have the bound $\E[\theta(S, \bx) ~|~ \bx] \geq  \lambda_{\min}\left( \E\left[\bM_{[S]}^{-1}\right]\right)$ in the smooth case from \eqref{eq:prop_sdna_sampling}. Substituting into \eqref{eq:convergence_non-convex_condition}, we get the rate
\[K \geq \frac{\xi(\bx^0)}{\epsilon \lambda_{\min}( \E[\bM_{[S]}^{-1}])}\log\left( \frac{\xi(\bx^0)}{\epsilon}\right) \qq{\Rightarrow} \left( (\E[\xi(\bx^K)] \leq \epsilon) \vee (\exists k \in [K] : \lambda(\bx^k) \leq \epsilon)\right)\]
for smooth objectives. In the non-smooth case we can use the bound \eqref{eq:proportion_non-smooth_minibatch} given by $\E[\theta(S, \bx) ~|~ \bx] \geq \tau/nL_{\tau}$, where $L_\tau$ is given by \eqref{eq:L_general}. It follows that we get the rate \[K \geq \frac{\xi(\bx^0) n L_{\tau}}{\tau \epsilon}\log\left( \frac{\xi(\bx^0)}{\epsilon}\right) \qq{\Rightarrow} \left( (\E[\xi(\bx^K)] \leq \epsilon) \vee (\exists k \in [K] : \lambda(\bx^k) \leq \epsilon)\right)\]
for non-smooth functions.

\paragraph{Greedy coordinate descent.} In the smooth case we have the bound  
\eqref{eq:prop_lower_bound_rcd_greedy}, i.e., $\theta(S_k,\bx^k) \geq c \eqdef 1/\sum_{i=1}^n M_{ii}$ . This leads to the rate
\[K \geq \frac{\xi(\bx^0)\sum_{i=1}^nM_{ii}}{\epsilon}\log\left( \frac{\xi(\bx^0)}{\epsilon}\right) \qq{\Rightarrow} \left( (\xi(\bx^K) \leq \epsilon) \vee (\exists k \in [K] : \lambda(\bx^k) \leq \epsilon)\right)\]
for smooth functions. In the proximal case, we have a bound \eqref{eq:proportion_non-smooth_serial_greedy} given by $\theta(S, \bx^k) \geq c \eqdef 1/n \max_i\{M_{ii}\}$ which leads to the rate
\[K \geq \frac{\max_{i}\{M_{ii}\}\xi(\bx^0)}{\epsilon}\log\left( \frac{\xi(\bx^0)}{\epsilon}\right) \qq{\Rightarrow} \left( (\xi(\bx^K) \leq \epsilon) \vee (\exists k \in [K] : \lambda(\bx^k) \leq \epsilon)\right)\]
for non-smooth functions.

\paragraph{Greedy minibatches.}

In the smooth case we have the bound \eqref{eq:prop_minibatch_greedy} which leads to the rate
\[K \geq \frac{\xi(\bx^0)}{\epsilon \lambda_{\min}( \E[\bM_{[S]}^{-1}])}\log\left( \frac{\xi(\bx^0)}{\epsilon}\right) \qq{\Rightarrow} \left( (\xi(\bx^K) \leq \epsilon) \vee (\exists k \in [K] : \lambda(\bx^k) \leq \epsilon)\right)\]

In the non-smooth case we have the bound on the proportion function given by \eqref{eq:proportion_non-smooth_minibatch_greedy} which leads to the rate
\[K \geq \frac{\xi(\bx^0) n L_{\tau}}{\tau \epsilon}\log\left( \frac{\xi(\bx^0)}{\epsilon}\right) \qq{\Rightarrow} \left( (\xi(\bx^K) \leq \epsilon) \vee (\exists k \in [K] : \lambda(\bx^k) \leq \epsilon)\right)\]
Uniform minibatch coordinate descent has bounds of the same form, except that the expectation is missing due to the deterministic nature of the method.

Also, both of the above rate are new, as greedy minibatches was introduced as a novel sampling approach.

\section{Experiments} \label{sec:experiments}

In this section we show results of sample numerical experiments. We will focus on showcasing the theory of the general non-convex optimization methods introduced in Section~\ref{sec:general_nonconvex}. 

\subsection{Setup}

For our experiments, we consider a function defined as in \eqref{eq:prox_problem} with $f$ defined as
\begin{equation} \label{eq:experiments_f}
f(\bx) \qq{\eqdef} \frac{1}{2m} \|\mathbf{A}\bx - \mathbf{b}\|^2 + \frac{1}{m} \cos( \langle \mathbf{c}, \bx \rangle)
\end{equation}
with $\bA \in \R^{m \times n}$ and $\bb, \mathbf{c} \in \R^n$, and the function $g$ defined as \begin{equation} \label{eq:experiments_g}
g(\bx) \qq{\eqdef} \lambda \|\bx\|_1.
\end{equation}
While it is not motivated by any specific problem, it is clearly non-convex and we can easily control it to observe the behavior of the proposed method. 

\subsection{Global convergence of serial coordinate descent}

In the first part, we consider the setup defined in the above section using \eqref{eq:experiments_f} and \eqref{eq:experiments_g} with $m = 1000$ and $n = 100$. We generate $\bA \in \R^{m \times n}$ as a random matrix with fixed singular values linearly spaced between $\frac{1}{m}$ and $1$. The vector $\bb \in \R^m$ is set to $\bb = \bA \by$ for a vector $\by \in \R^n$ randomly generated from a normalized Gaussian distribution. Similarly, $\mathbf{c}$ is also randomly generated from a normalized Gaussian. We consider two different problems, based on the value of $\lambda$. We have a smooth problem for $\lambda = 0$ and a non-smooth problem for $\lambda = \frac{1}{2m}$.

We measure the performance of the various coordinate descent approaches described in Section~\ref{sec:general_non-convex_applications}. The convergence behaviors corresponding to the smooth and non-smooth setup can be found in Figure~\ref{fig:serial_convergences}, on the left and right plots, respectively.

To make sure that the functions being optimized are indeed non-convex, we plotted a 1-dimensional slice of the function being optimized around its optimum. These plots can be found on Figure~\ref{fig:serial_reliefs}, for the smooth and non-smooth version respectively.

\begin{figure}[bt]
\begin{subfigure}{0.49\linewidth}
\includegraphics[width=\linewidth]{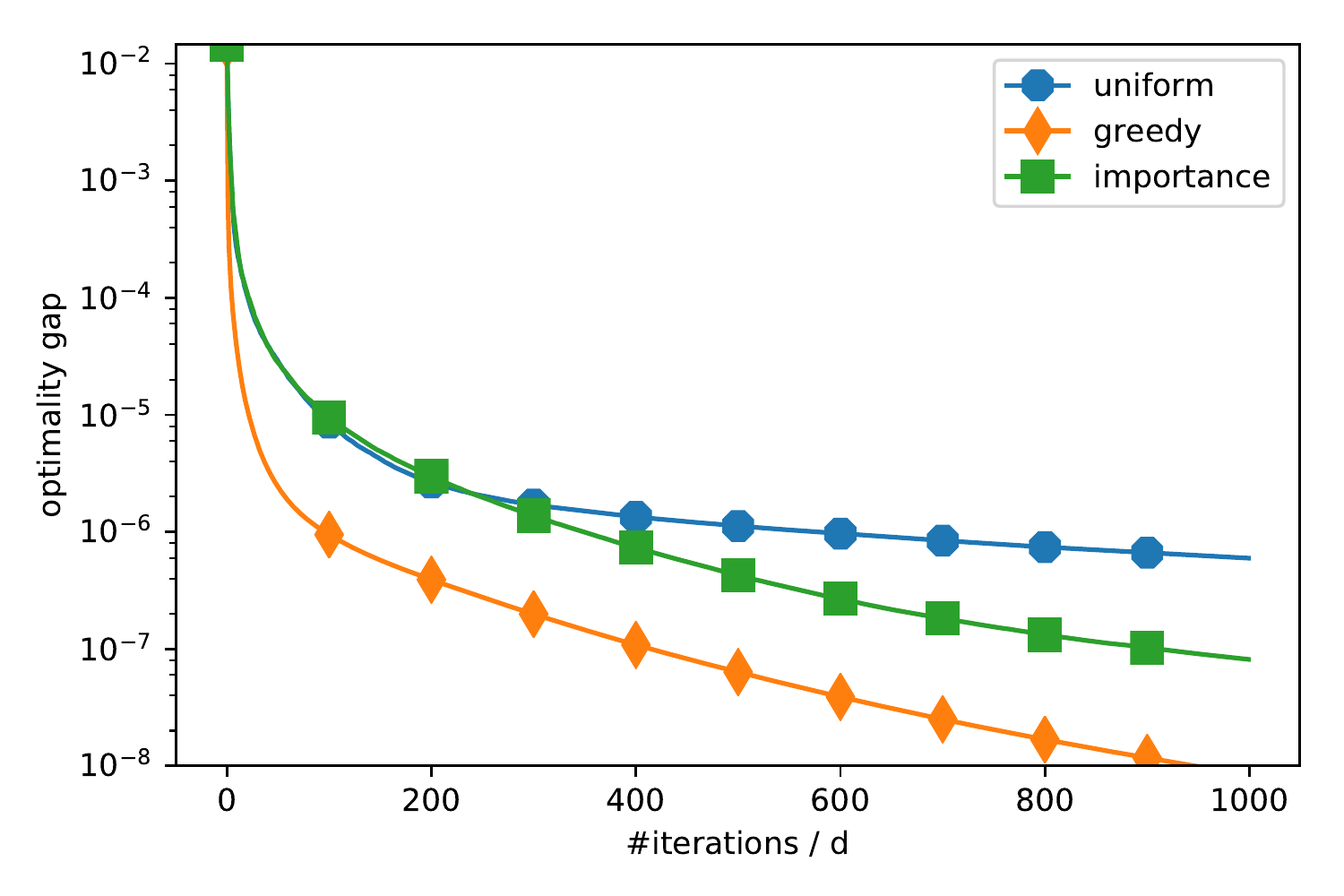}
\end{subfigure}
\begin{subfigure}{0.49\linewidth}
\includegraphics[width=\linewidth]{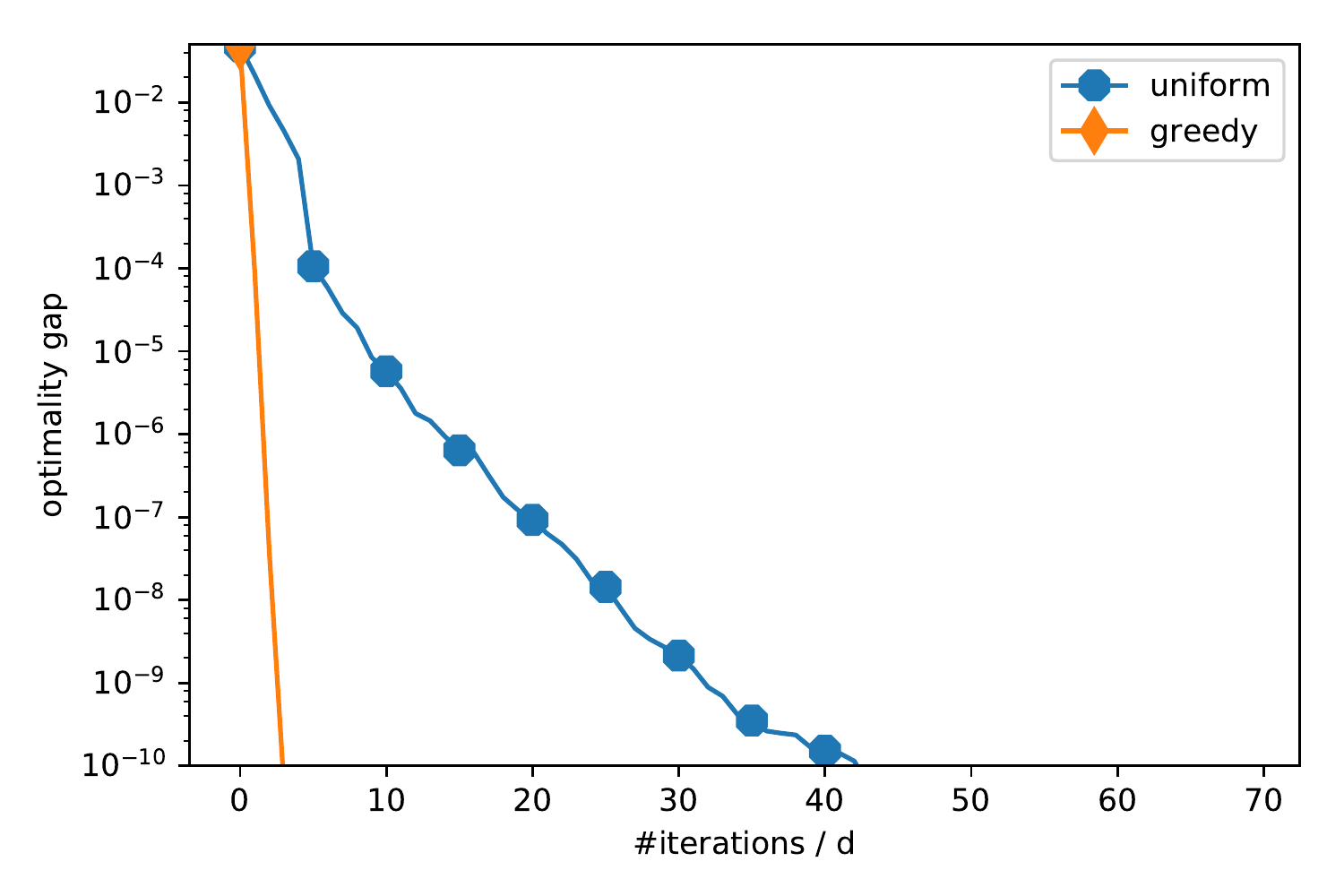}
\end{subfigure}
\caption{Plots of the convergences of various sampling methods for the smooth (left) and non-smooth (right) experiment.}
\label{fig:serial_convergences}
\end{figure}

\begin{figure}[bt]
\begin{subfigure}{0.49\linewidth}
\includegraphics[width=\linewidth]{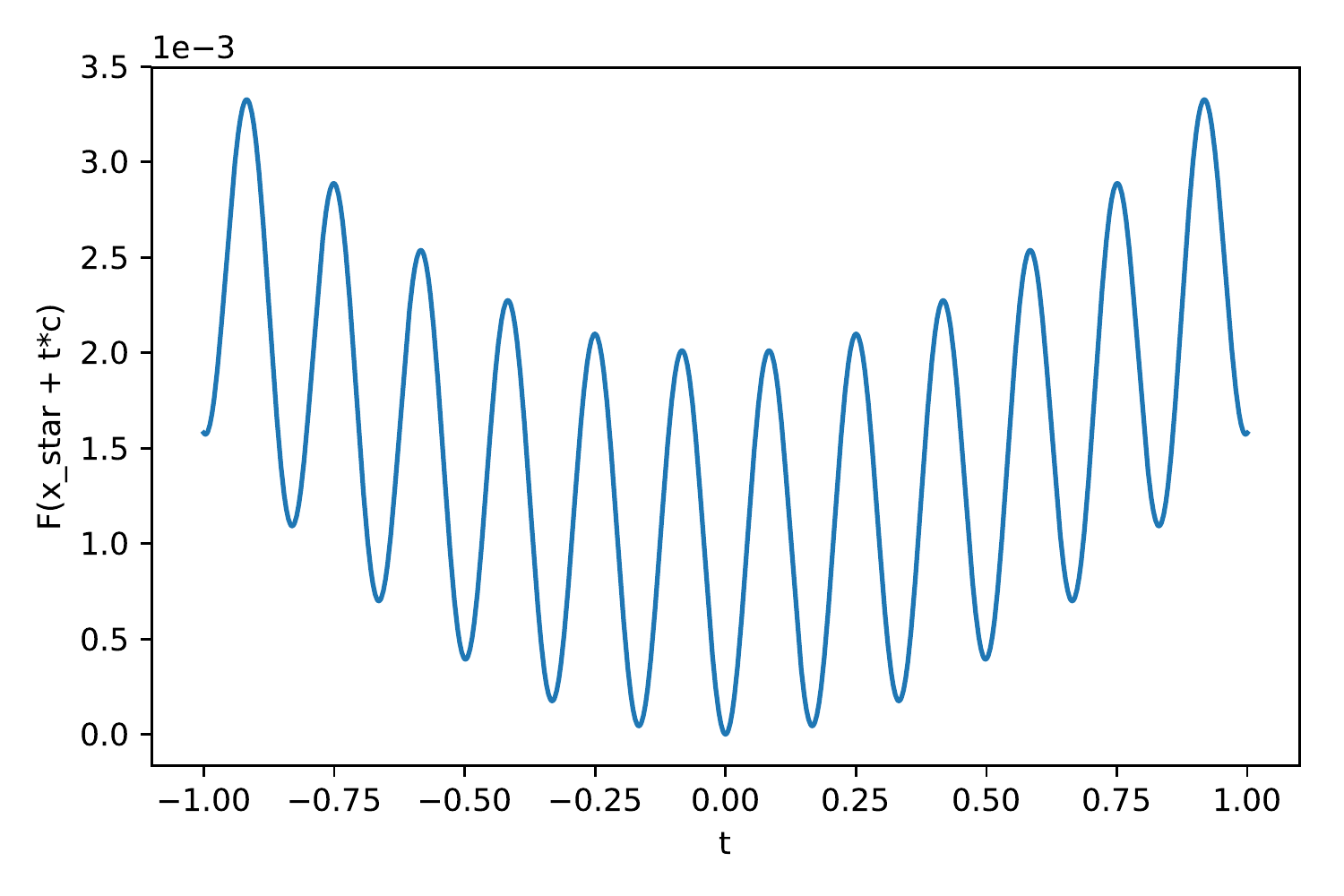}
\end{subfigure}
\begin{subfigure}{0.49\linewidth}
\includegraphics[width=\linewidth]{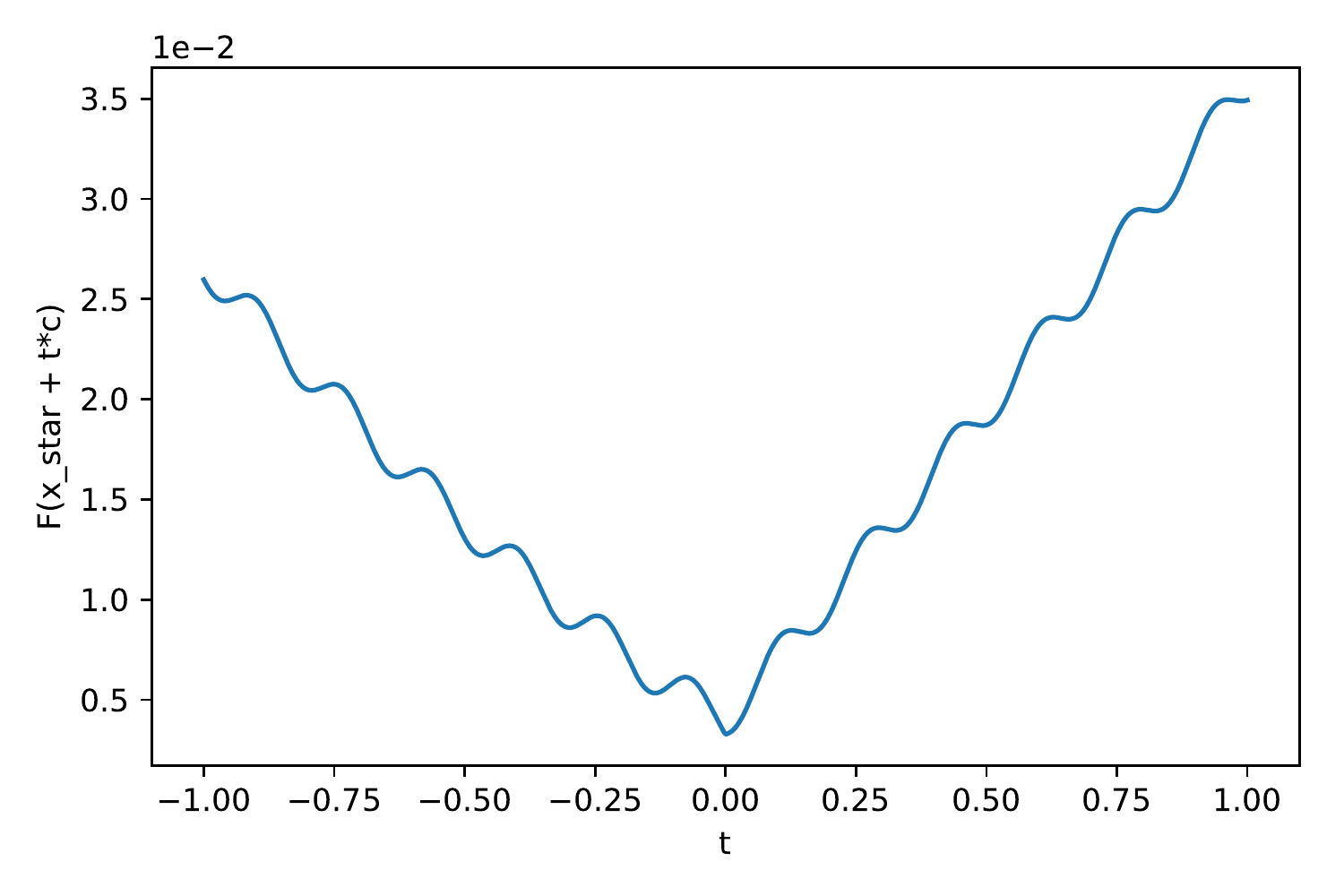}
\end{subfigure}
\caption{Plots of the functions in a 1-dimensional slice around the optimum for the smooth (left) and non-smooth (right) experiment.}
\label{fig:serial_reliefs}
\end{figure}

\subsection{Local convergence of the gradient}

The theory of the non-convex case in Theorem~\ref{thm:non-convex_case} does not always guarantee convergence to the global optimum, but it at least guarantees a convergence of the magnitude of the gradient. To showcase this scenario, we focused on a 1-dimensional instance of the smooth problem defined in \eqref{eq:experiments_f}. We set $\bA = [1]$, $\bb = \tfrac{\pi}{c}$ and we have chosen $c$ such that the function $f$ has a flat inflection point ($c \approx 2.15$). It is trivial to show, that the optimal value is 0 and it is achieved for $x = \tfrac{\pi}{c} \approx 1.46$. The shape of the function around the optimum can be found on the left plot of Figure~\ref{fig:1d_example}.

We used a 1-dimensional gradient descent method for the convergence and we reported on three quantities: The function suboptimality (\textit{fx}), the magnitude of the gradient (\textit{dfx}) and the rate predicted by the theory (\textit{rate}). Theory states that either the function suboptimality or the magnitude of the gradient has to be below the predicted rate, which is what we observe in the right plot of Figure~\ref{fig:1d_example} as well. Note that the theory focuses on worst case bounds, which is possibly the case why the difference between the rate and the magnitudes is so huge.

\begin{figure}[bt]
\begin{subfigure}{0.49\linewidth}
\includegraphics[width=\linewidth]{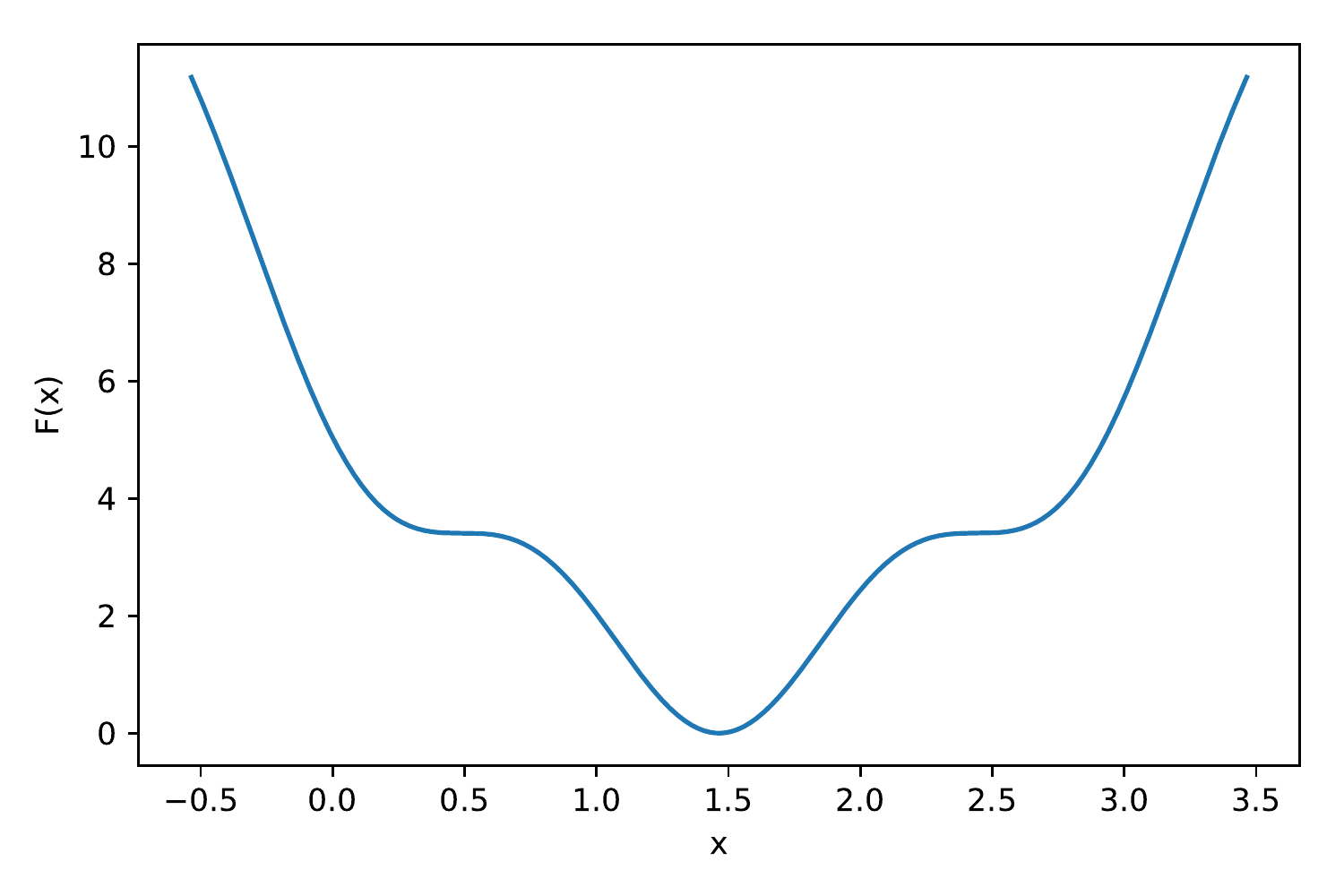}
\end{subfigure}
\begin{subfigure}{0.49\linewidth}
\includegraphics[width=\linewidth]{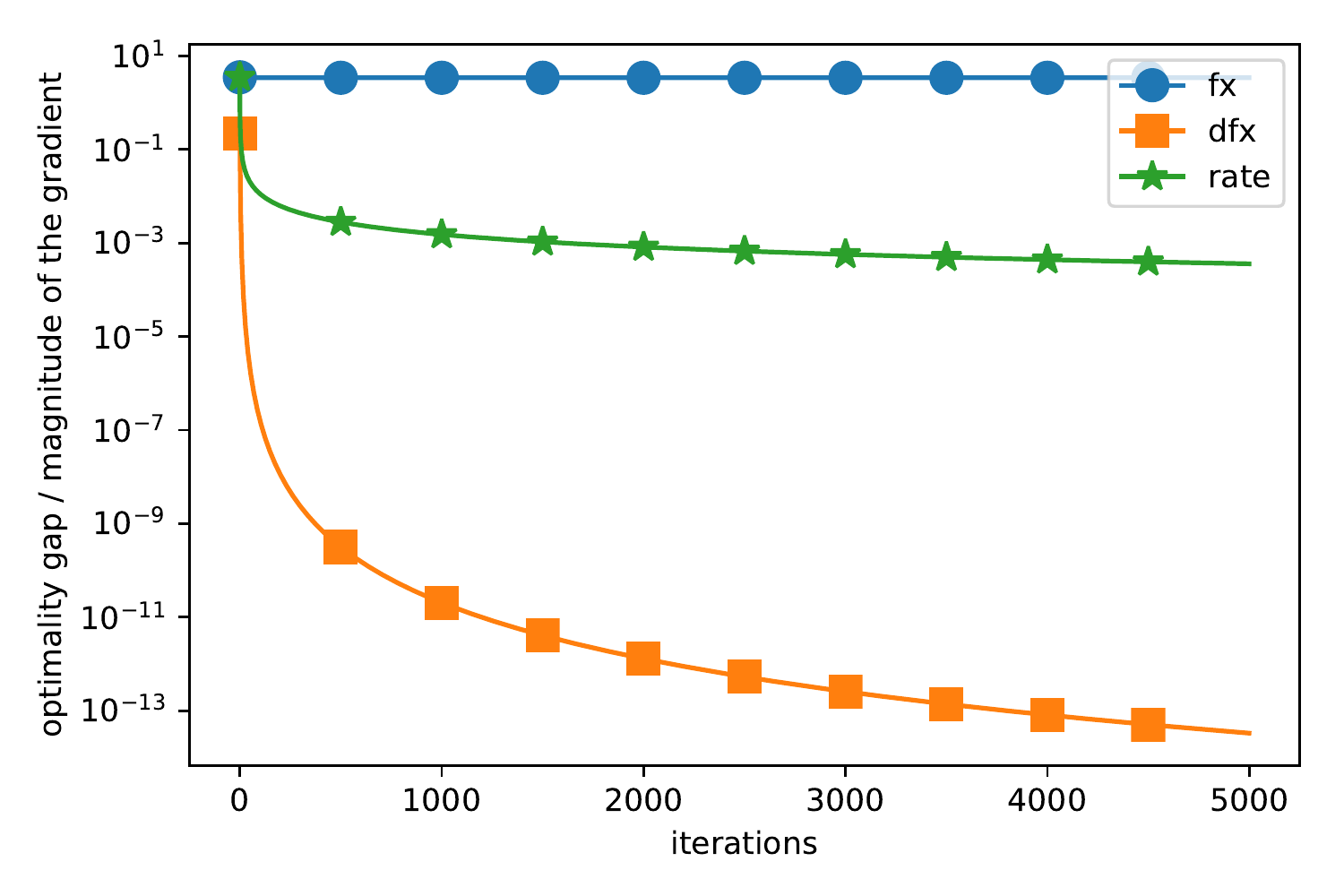}
\end{subfigure}
\caption{Plot of a one-dimensional function with a plateau being optimized (left) and the convergence of the magnitude of the gradient descent with its corresponding rate (right).}
\label{fig:1d_example}
\end{figure}

\bibliographystyle{natbib}
\bibliography{biblio}

\newpage
\appendix

\section{Some basic properties of WPL functions} \label{Appendix-WPL}

Here we establish some basic properties of smooth WPL functions.

\begin{proposition} Assume $f\in \cW_{PL}(\mu)$ for some $\mu>0$. Then
\begin{itemize}
\item[(i)] $\cX^* = \{\bx \in \R^n \;:\; \nabla f(\bx) = 0\}$
\item[(ii)] $\cX^* = \{\bx \in \R^n \;:\; \xi(\bx) = 0\}$ 
\item[(iii)] If $\nabla f(\bx)$ is continuous, then $\cX^*$ is convex.
\item[(iv)] Assume $\nabla f(\bx)$ is continuous. Then $f$ is non-decreasing on all rays emanating from $\bx^*$. That is, $\psi(t) \eqdef f(\bx^* + t(\bx-\bx^*))$ is non-decreasing on $t\geq 0$ for all $\bx\in \R^n$.
\end{itemize}
\end{proposition}

\begin{proposition} Assume that $f$ has continuous gradient. If there exists a constant $c>0$ such that 
\[\int_{0}^1 \|\nabla f(\bx^* + t(\bx-\bx^*))\|	\; dt \qq{\leq} c \|\nabla f(\bx)\|, \qquad \bx\in \R^n,\]
then $f\in \cW_{PL}(\tfrac{1}{c^2})$.
\end{proposition}
\begin{proof}  Using the fundamental theorem of calculus, Cauchy-Schwartz inequality, and then  applying the assumption, we get \begin{eqnarray*} \xi(\bx) = f(\bx) - f(\bx^*) & = &  \int_{0}^1 \langle \nabla f(\bx^* + t(\bx-\bx^*)), \bx-\bx^* \rangle \; dt\\
&\leq & \int_{0}^1 \|  \nabla f(\bx^* + t(\bx-\bx^*))\| \cdot \| \bx-\bx^*\| \; dt \\
&\leq & c \| \nabla f(\bx)\| \|\bx-\bx^*\| .
\end{eqnarray*}
\end{proof}

Let us shed light on the above result. If $f$ is convex, then the directional derivative $\psi(t)\eqdef \langle \nabla f(\bx^* + t(\bx-\bx^*)), \bx-\bx^* \rangle$ is an increasing function of $t$, and hence can be bounded above on $[0,1]$ by $\psi(1) = \langle \nabla f(\bx), \bx-\bx^* \rangle \leq \|\nabla f(\bx)\| \|\bx-\bx^*\|$. It follows that $f\in \cW_{PL}(1)$, which we already know.

\begin{theorem} The following hold: 
\begin{enumerate}
\item If $\mu_1 \geq \mu_2$, then $\cW_{PL}(\mu_1)\subseteq \cW_{PL}(\mu_2)$.

\item If $f$ is convex, then $f\in \cW_{PL}(\mu)$ for all $\mu \leq 1$.
\item If $f\in \cW_{PL}(\mu)$, then $af+b \in \cW_{PL}(\mu)$ for all $a\geq 0$ and $b\in \R$. 
\item If $f\in \cC^1(L)$, then $\cS_{PL}(\mu)\subseteq \cW_{PL}(4\mu/L)$.
\item Let $f\in \cC^1(L)$, fix $\bx^*\in \cX^*$, and assume that there exists a constant $c>0$ such that $\|\bx-\bx^*\| \leq c \|\nabla f(\bx)\|$ for all $\bx\in \R^n$. Then $f\in \cW_{PL}(\tfrac{4}{L^2 c^2})$.
\item Assume that there exists a constant $c>0$ such that $\|\nabla f(\bx)\| \leq c$ for all $\bx \in \R^n$. If $f\in \cW_{PL}(\mu)$, then $\xi(\bx) \leq \frac{c}{\sqrt{\mu}}\|\bx - \bx^*\|$ for all $\bx \in \R^n$. That is, $f$ is Lipschitz on each ray emanating from $\bx^*$ with Lipschitz constant $\frac{c}{\sqrt{\mu}}$.
\item Assume $f\in \cW_{PL}(\mu)$ with $\mu>0$. If 
\[\frac{\langle \nabla f(\bx), \bx-\bx^* \rangle}{\|\nabla f(\bx)\|\|\bx-\bx^*\|} \qq{\geq} \frac{1}{\sqrt{\mu}},\]
then $f$ satisfies the restricted convexity property: 
\[f(\bx^*) \qq{\geq} f(\bx) + \langle \nabla f(\bx), \bx^*-\bx \rangle, \qquad \bx\in \R^n.\]
\end{enumerate}
\end{theorem}

\begin{proof}

\begin{enumerate}
\item Obvious.
\item This was established in the introduction for $\mu=1$.  It only remains to apply 1) to conclude that 2) holds for all $\mu \leq 1$.  
\item Obvious. 
\item If $f\in \cS_{PL}(\mu)$, then
\[\frac{\|\nabla f(\bx)\|^2 \cdot \|\bx-\bx^*\|^2}{\xi^2(\bx)} \qq{\overset{\eqref{eq:s09h9hf3}}{\geq}} \frac{2 \mu \|\bx - \bx^*\|^2}{\xi(\bx)} \qq{\overset{\eqref{eq:L-smoothness}}{\geq}}  \frac{4 \mu}{L}.\]

\item We have $\xi(\bx) \leq \frac{L}{2}\|\bx - \bx^*\|^2$ using smoothness. Combining this with the assumption $\|\bx - \bx^*\| \leq c\|\nabla f(\bx)\|$ from the claim we will prove $f \in \cW_{PL}(\frac{4}{L^2 c^2})$ from the definition \eqref{eq:weak_PL_def} \[\frac{2}{Lc}\xi(\bx) \qq{\leq} \frac{1}{c}\|\bx - \bx^*\|^2 \qq{\leq}  \| \nabla f (\bx)\| \cdot \|\bx - \bx^* \|.\]
\item Directly using the definition of $f \in \cW_{PL}(\mu)$ with the assumption $\|\nabla f(\bx)\| \leq c$ we get \[\xi(\bx) \qq{\stackrel{\eqref{eq:weak_PL_def}}{\leq}} \frac{1}{\sqrt{\mu}} \|\nabla f(\bx)\| \|\bx - \bx^*\| \qq{\leq} \frac{c}{\sqrt{\mu}} \|\bx - \bx^*\|.\]
\item Combining the definition of $f \in \cW_{PL}(\mu)$ \eqref{eq:weak_PL_def} with the assumption from the claim we get that \[\sqrt{\mu}\xi(\bx) \qq{\stackrel{\eqref{eq:weak_PL_def}}{\leq}} \|\nabla f(\bx)\| \cdot \| \bx - \bx^*\| \qq{\leq} \sqrt{\mu} \langle \nabla f(\bx),  \bx - \bx^* \rangle.\]
Dividing both sides by $\sqrt{\mu}$ and adding $f(\bx^*)$ we get the restricted convexity.
\end{enumerate}
\end{proof}

\newpage
\section{Proofs}
\subsection{Proof of Lemma~\ref{lem:onestepbound}} \label{sec:proof_onestepbound}

\begin{proof}
From the definition of Algorithm~\ref{alg:general} we have that $\bx^+ = \bx + \bu_{[S]}^*$, where \begin{equation} \label{eq:prf_onestepbound_ustar}
\bu^* \qq{\eqdef} \argmin_{\bu \in \R^n} \left\{ \langle \nabla_{[S]} f(\bx), \bu \rangle + \frac{1}{2}\bu^\top \bM_{[S]} \bu + \sum_{i \in S} \left[ g_i(x_i + u_i) - g_i(x_i) \right] \right\}
\end{equation}

It follows that \begin{eqnarray*}
\xi(\bx^+) &=& \xi(\bx + \bu^*) \\
&\stackrel{\eqref{eq:optimality_gap}}{=}& F(\bx + \bu^*) - F(\bx^*) \\
&\stackrel{\eqref{eq:prox_problem}}{=}& f(\bx + \bu^*_{[S]}) + g(\bx + \bu^*_{[S]}) - F(\bx^*)\\
&\stackrel{\eqref{eq:Msmoothness}}{\leq} & f(\bx) + \langle \nabla f(\bx), \bu^*_{[S]} \rangle + \frac{1}{2} (\bu^*_{[S]})^\top \bM \bu^*_{[S]} + g(\bx + \bu_{[S]}^*) - F(\bx^*) \\
&\stackrel{\eqref{eq:prox_problem}}{=} & F(\bx) + \langle \nabla f(\bx), \bu^*_{[S]} \rangle + \frac{1}{2} (\bu^*_{[S]})^\top \bM \bu^*_{[S]} + g(\bx + \bu_{[S]}^*) - g(\bx) - F(\bx^*) \\
&\stackrel{\eqref{eq:optimality_gap}}{=} & \xi(\bx) + \langle \nabla f(\bx), \bu^*_{[S]} \rangle + \frac{1}{2} (\bu^*_{[S]})^\top \bM \bu^*_{[S]} + g(\bx + \bu_{[S]}^*) - g(\bx) \\
&\stackrel{\eqref{eq:separability}}{=} & \xi(\bx) + \langle \nabla f(\bx), \bu^*_{[S]} \rangle + \frac{1}{2} (\bu^*_{[S]})^\top \bM \bu^*_{[S]} + \sum_{i \in S} \left[ g_i(x_i + u_{i}^*) - g_i(x_i) \right] \\
&\stackrel{\eqref{eq:prf_onestepbound_ustar}}{=} & \xi(\bx) + \min_{\bu \in \R^n} \left\{ \langle \nabla_{[S]} f(\bx), \bu \rangle + \frac{1}{2}\bu^\top \bM_{[S]} \bu + \sum_{i \in S} \left[ g_i(x_i + u_i) - g_i(x_i) \right] \right\} \\
&\stackrel{\eqref{eq:proportion_function}}{\leq}& \xi(\bx) + \theta(S,\bx) \cdot \min_{\by \in \R^n} \left\{ \langle \nabla f(\bx), \by \rangle + \frac{L}{2}\|\by\|^2 + g(\bx + \by) - g(\bx) \right\} \\
&\stackrel{\eqref{eq:GPL_function}}{=}& [1 - \theta(S,\bx) \cdot \mu(\bx)] \cdot \xi(\bx)
\end{eqnarray*}
Note that the inequality in the one-to-last line might happen in the case, when $\bx \notin \mathcal{X}$ in the definition of the proportion function in \eqref{eq:proportion_function}.
Chaining up the resulting expressions from $\bx^K$ all the way back to $\bx^0$ proves the $K$-step bound.
\end{proof}

\subsection{Proof of Theorem~\ref{thm:GPL_convex_strongly_convexXXX}} \label{sec:proof_GPL_strongly_convex}
\begin{proof}
Using the result of Lemma~\ref{lem:technical_GPL}, we have the bound \eqref{eq:technical_lemma_GPL}. To get the result, we will use the bound $(a + b)^2 \geq 4ab$, which holds for any $a,b > 0$. Specifically, we will use it for $a = \xi(\bx)$ and $b = \tfrac{\lambda_F}{2}\|\bx - \bx^*\|^2$ in the expression in \eqref{eq:technical_lemma_GPL} to get 
\begin{align*}
- L \cdot \min_{\by \in \R^n} &\left\{ \langle \nabla f(\bx), \by \rangle + \frac{L}{2}\|\by\|^2 + g(\bx + \by) - g(\bx) \right\} \\ &\geq \xi(\bx) \cdot \min \left\{ \frac{L}{2}, \frac{L\lambda_F}{\lambda_F - \lambda_f + L} \right\},
\end{align*}
which is the claimed result.
\end{proof}

\subsection{Proof of Theorem~\ref{thm:SPLg}} \label{sec:proof_SPLg_general}
\begin{proof}
The one-step bound \eqref{eq:onestepbound} combined with the definition of the class $\cS_{PL}^g$ \eqref{eq:weak_PL_def} gives
\begin{align*}
\xi(\bx^{k+1}) \qq{\stackrel{\eqref{eq:onestepbound}}{\leq}} \left( 1 - \theta(S_k, \bx^k) \cdot \mu(\bx^k) \right) \cdot \xi(\bx^k) \qq{\stackrel{\eqref{eq:weak_PL_def}}{\leq}}  (1 - \theta(S_k, \bx^k) \cdot \mu) \cdot \xi(\bx^k).
\end{align*}
Now, by taking full expectation over the whole sampling procedure on both sides and using the definition of $\mu_k$ \eqref{eq:SPLg_muk} we get
\begin{eqnarray*}
\E [\xi(\bx^{k+1})] & \leq & \E\left[\left (1 - \theta(S_k, \bx^k) \mu \right) \cdot \xi(\bx^k)\right] \\
& = & \E[\xi(\bx^k)] - \mu \E[\theta(S_k, \bx^k) \cdot \xi(\bx^k)] \\
 & \stackrel{\eqref{eq:SPLg_muk}}{=} &  \left( 1 - \mu_k \right) \cdot \E [\xi(\bx^k)].
\end{eqnarray*}
To establish the convergence rate \eqref{eq:SPLg_convergence_rate}, we use the estimate $1 - s \leq e^{-s}$ to get $\E[\xi(\bx^{k+1})] \leq e^{-\mu_k} \E[\xi(\bx^{k})],$
which we can simply chain together repeatedly to get \[\E[\xi(\bx^{k+1})] \qq{\leq} e^{- \sum_{k=0}^{K-1} \mu_k} \xi(\bx^0).\]
Setting the right-hand side less or equal to $\epsilon$ and rearranging we finally get \eqref{eq:SPLg_convergence_rate}.
\end{proof}

\subsection{Proof of Theorem~\ref{thm:GPL_convex_weakly_convexXXX}} \label{sec:proof_GPL_convex}

\begin{proof}
Using the result of Lemma~\ref{lem:technical_GPL}, we have the bound \eqref{eq:technical_lemma_GPL}. Plugging $\lambda_F = \lambda_f = 0$ into the expression in \eqref{eq:technical_lemma_GPL} we get 
\begin{align*}
- L \cdot \min_{\by \in \R^n} &\left\{ \langle \nabla f(\bx), \by \rangle + \frac{L}{2}\|\by\|^2 + g(\bx + \by) - g(\bx) \right\} \\ &\geq \xi^2(\bx) \cdot \min \left\{ \frac{L}{2 \xi(\bx)}, \frac{1}{2 \|\bx - \bx^*\|^2} \right\},
\end{align*}
which is the first part of the claimed result. As for the second part, we can directly bound $\xi(\xi) \leq \xi(\bx^0)$, as these are the only pairs of $\bx, \bx^0$ we need to consider according to Definition~\ref{def:weaklyPL-general}. Similarly, as the function value is bounded, the quantities $\|\bx - \bx^*\|$ can be upper bounded by the largest distance in the level set of $f(\bx^0)$, which is given by $R$ in \eqref{eq:weakly_PL_R_convergence}. Combining these arguments, we get that \[\mu(\bx) \qq{\geq} \xi(\bx) \cdot \min \left\{ \frac{L}{2 \xi(\bx)}, \frac{1}{2 \|\bx - \bx^*\|^2} \right\} \qq{\geq} \xi(\bx) \cdot \min \left\{ \frac{L}{2 \xi(\bx^0)}, \frac{1}{2 R^2} \right\},\]
which is the claimed result.
\end{proof}

\subsection{Proof of Theorem~\ref{thm:WPLg}} \label{sec:proof_WPLg_general}

\begin{proof}
Combining the one-step bound \eqref{eq:onestepbound} with the definition of the class $\cW_{PL}^{g}$ \eqref{eq:weakly_PL_general} we get that
\begin{align*}
\xi(\bx^{k+1}) &\qq{\stackrel{\eqref{eq:onestepbound}}{\leq}}  [1 - \theta(S_k, \bx^k) \cdot \mu(\bx^k)] \cdot \xi(\bx^k)  \\
&\qq{\stackrel{\eqref{eq:weakly_PL_general}}{\leq}} [1 - \theta(S_k, \bx^k) \cdot \rho(\bx^0) \cdot \xi(\bx^k)] \cdot \xi(\bx^k)
\end{align*}
 Now, taking full expectation over the whole sampling process on both sides and using the definition of $\mu_k$ \eqref{eq:WPLg_muk} we get
\begin{eqnarray*}
\E [\xi(\bx^{k+1})] &\leq & \E\left[\left (1 - \theta(S_k, \bx^k) \cdot \rho(\bx^0) \cdot \xi(\bx^k)\right) \cdot \xi(\bx^k)\right] \\
& = & \E[\xi(\bx^k)] - \rho(\bx^0) \E[\theta(S_k, \bx^k) \cdot (\xi(\bx^k))^2] \\
&\stackrel{\eqref{eq:WPLg_muk}}{=} & \left( 1 - \mu_k \E [\xi(\bx^{k})]\right) \cdot \E [\xi(\bx^k)].
\end{eqnarray*}
Observe, that $\{\mu_k\}_{k=0}^{K}$ and $\{\E[\xi(\bx^k)]\}_{k=0}^K$ are both positive scalars, therefore we can use Lemma~\ref{lem:onestepbound_xi} to get the bound \[\E[\xi(\bx^K)] \qq{\leq} \frac{\xi(\bx^0)}{1 + \xi(\bx^0) \sum_{k=0}^{K-1} \mu_k}.\] Putting the right-hand side less than $\epsilon$ and rearranging leads to the claimed result \eqref{eq:WPLg_convergence_rate}.
\end{proof}

\subsection{Proof of Theorem~\ref{thm:non-convex_case}} \label{sec:proof_non-convex_case}

\begin{proof}
If part $(i)$ from the claim holds, we are done. Now assume on the contrary, that $(i)$ does not hold, i.e., \begin{equation} \label{eq:thm:non-convex_prf1}
- L \cdot  \min_{\by \in \R^n} \left\{ \langle \nabla f(\bx^k), \by \rangle + \frac{L}{2}\|\by\|^2 + g(\bx^k + \by) - g(\bx^k) \right\} \qq{\geq} \epsilon
\end{equation} for all $k \in \{0, \dots, K-1\}$. It follows, that 
\begin{align*}
\mu(\bx^k) \qq{\stackrel{\eqref{eq:GPL_function_smooth}}{=}} \frac{-L \cdot \min_{\by \in \R^n} \left\{ \langle \nabla f(\bx^k), \by \rangle + \frac{L}{2}\|\by\|^2 + g(\bx^k + \by) - g(\bx^k) \right\}}{\xi(\bx^k)} \qq{\stackrel{\eqref{eq:thm:non-convex_prf1}}{\geq}} \frac{\epsilon}{\xi(\bx^k)}.
\end{align*}
for all $k \in \{0, \dots, K-1\}$.
Using the result from Lemma~\ref{lem:onestepbound}, we have that $\xi(\bx^k) \leq \xi(\bx^0)$ for all $k$, which we can use to further bound \begin{equation} \label{eq:thm:non-convex_prf2}
\mu(\bx^k) \qq{\geq} \frac{\epsilon}{\xi(\bx^0)}.
\end{equation}
Using \eqref{eq:onestepbound} combined with the above result \eqref{eq:thm:non-convex_prf2} we get
\begin{align*}
\xi(\bx^K) &\qq{\stackrel{\eqref{eq:onestepbound}}{\leq}} \left[1 - \mu(\bx^{K-1}) \cdot \theta(S_{K-1}, \bx^{K-1})\right] \xi(\bx^{K-1})  \\
&\qq{\stackrel{\eqref{eq:thm:non-convex_prf2}}{\leq}} \left[ 1 - \frac{\epsilon \theta(S_{K-1}, \bx^{K-1})}{\xi(\bx^0)} \right] \xi(\bx^{K-1})
\end{align*}
Taking the expectation over the whole sampling procedure on both sides and using the definition of $\mu_k$ in \eqref{eq:non-convex_muk} we get \begin{eqnarray}
\E[\xi(\bx^K)] & \leq & \E \left[ \left( 1- \frac{\epsilon \theta(S_{K-1}, \bx^{K-1})}{\xi(\bx^0)} \right) \xi(\bx^{K-1})\right] \nonumber \\
& = & \E[\xi(\bx^{K-1})] - \frac{\epsilon}{\xi(\bx^{0})} \E[\theta(S_{K-1},\bx^{K-1}) \cdot \xi(\bx^{K-1})] \nonumber \\
 & \stackrel{\eqref{eq:non-convex_muk}}{=} & \left(1 - \frac{\epsilon \mu_{K-1}}{\xi(\bx^0)} \right) \E[\xi(\bx^{K-1})] \label{eq:thm:non-convex_prf3}
\end{eqnarray}
 Combining the above inequality \eqref{eq:thm:non-convex_prf3} with $(1 - z) \leq \exp(-z)$ repeatedly, we get 
\begin{eqnarray*}
\E[\xi(\bx^K)] & \stackrel{\eqref{eq:thm:non-convex_prf3}}{\leq} & \left(1 - \frac{\epsilon \mu_{K-1}}{\xi(\bx^0)}\right) \E[\xi(\bx^{K-1})] \\
& \leq & \exp \left(- \frac{\epsilon \mu_{K-1}}{\xi(\bx^0)} \right) \E[\xi(\bx^{K-1})] \\
& \stackrel{\eqref{eq:thm:non-convex_prf3}}{\leq} & \dots \\ 
&\stackrel{\eqref{eq:thm:non-convex_prf3}}{\leq} & \exp \left(- \frac{\epsilon \sum_{k=0}^{K-1}\mu_{k}}{\xi(\bx^0)} \right) \xi(\bx^0) \\
&\stackrel{\eqref{eq:convergence_non-convex_condition}}{\leq} &  \epsilon,
\end{eqnarray*}
where the last line follows from comparing the logarithms of both sides. This proves $(ii)$.
\end{proof}

\section{Technical Lemmas}

\subsection{Lemma ~\ref{lem:technical_GPL}}
\begin{lemma} \label{lem:technical_GPL}
Let $f$ be $\lambda_f$-strongly convex \eqref{eq:strongly_convex_smooth} with $\lambda_f \geq 0$ and $F$ be $\lambda_F$-strongly convex \eqref{eq:strongly_convex} with $\lambda_F \geq 0$. Then 
\begin{align}
- \min_{\by \in \R^n} \bigg\{ \langle \nabla f(\bx), &\by \rangle + \frac{L}{2}\|\by\|^2 + g(\bx + \by) - g(\bx) \bigg\} \nonumber \\
&\geq \min \left\{ \frac{1}{2}\xi(\bx),  \frac{\left(\xi(\bx) + \frac{\lambda_F}{2}\|\bx - \bx^*\|^2\right)^2}{2(\lambda_F - \lambda_f + L)\|\bx - \bx^*\|^2} \right\}. \label{eq:technical_lemma_GPL}
\end{align}
\end{lemma}
\begin{proof}
Let
\begin{equation} \label{eq:prf_GBP_generalized_convex_beta}
\beta \qq{=} \min\left\{1, \frac{\xi(\bx) + \frac{\lambda_F}{2}\|\bx - \bx^*\|^2}{(\lambda_F - \lambda_f + L)\|\bx - \bx^*\|^2}\right\}.
\end{equation}
Observe, that $\beta = 1$ implies, that \[(\lambda_F - \lambda_f + L)\|\bx - \bx^*\|^2 \qq{\leq} \xi(\bx) + \frac{\lambda_F}{2}\|\bx - \bx^*\|^2,\]
from which it follows that \begin{equation} \label{eq:prf_GBP_generalized_convex_beta_1}
\frac{\lambda_f - L}{2}\|\bx-\bx^*\|^2 \qq{\geq} -\frac{1}{2}\xi(\bx) +  \frac{\lambda_F}{4}\|\bx - \bx^*\|^2 \qq{>} -\frac{1}{2}\xi(\bx)
\end{equation} Now, it follows that 
\begin{eqnarray*}
- \min_{\by \in \R^n} & & \left\{ \langle \nabla f(\bx), \by \rangle + \frac{L}{2}\|\by\|^2 + g(\bx + \by) - g(\bx) \right\} \nonumber \\ 
&\stackrel{\eqref{eq:prox_problem}}{=}& F(\bx) - \min_{\by \in \R^n}\left\{ f(\bx) + \langle \nabla f(\bx), \by \rangle + \frac{L}{2}\|\by\|^2 + g(\bx + \by) \right\} \nonumber \\
&\stackrel{\eqref{eq:strongly_convex_smooth}}{\geq}& F(\bx) - \min_{\by \in \R^n}\left\{ f(\bx + \by) - \frac{\lambda_f}{2}\|\by\|^2 + \frac{L}{2}\|\by\|^2 + g(\bx + \by) \right\} \nonumber \\ 
&\stackrel{\eqref{eq:prox_problem}}{=}& F(\bx) - \min_{\by \in \R^n} \left\{ F(\bx + \by) + \frac{L - \lambda_f}{2}\|\by\|^2 \right\} \qquad \left(\mbox{let } \by = \beta(\bx^* - \bx)\right) \nonumber \\
&\geq & F(\bx) - F(\beta\bx^* + (1-\beta)\bx) - \frac{\beta^2(L - \lambda_f)}{2}\|\bx - \bx^*\|^2 \nonumber \\ 
&\stackrel{\eqref{eq:strongly_convex}}{\geq}& F(\bx) - \beta F(\bx^*) - (1-\beta) F(\bx) + \left(\frac{\lambda_F \beta(1-\beta)}{2} - \frac{\beta^2(L - \lambda_f)}{2} \right)\|\bx - \bx^*\|^2 \nonumber \\ 
&\stackrel{\eqref{eq:optimality_gap}}{=}& \beta \left( \xi(\bx) + \frac{\lambda_F}{2}\|\bx - \bx^*\|^2\right) - \beta^2 \frac{  (\lambda_F - \lambda_f + L)}{2}\|\bx - \bx^*\|^2 \nonumber \\
&\stackrel{\eqref{eq:prf_GBP_generalized_convex_beta}}{=}& \min \left\{\xi(\bx) + \frac{\lambda_f - L}{2}\|\bx - \bx^*\|^2, \frac{\left( \xi(\bx) + \frac{\lambda_F}{2}\|\bx - \bx^*\|^2\right)^2}{2(\lambda_F - \lambda_f + L)\|\bx - \bx^*\|^2} \right\} \nonumber \\
&\stackrel{\eqref{eq:prf_GBP_generalized_convex_beta_1}}{>}& \min \left\{ \frac{1}{2}\xi(\bx),  \frac{\left(\xi(\bx) + \frac{\lambda_F}{2}\|\bx - \bx^*\|^2\right)^2}{2(\lambda_F - \lambda_f + L)\|\bx - \bx^*\|^2} \right\} \nonumber. 
\end{eqnarray*}
\end{proof}
\subsection{Lemma ~\ref{lem:onestep_expectation}}

\begin{lemma} \label{lem:onestep_expectation}

Let $f$ be a given function and let $X, Y$ be random variables such that $\Prob(f(X) \geq 0) = 1$ and $\Prob(\E[Y ~|~ X] \geq 0) = 1.$ Additionally, let $c > 0 $ be a scalar such that 
\begin{equation} \label{eq:onestep_expectation_c}
c \qq{\leq} \E[Y ~|~ X]
\end{equation}
  Then it holds that \begin{equation} \label{eq:onestep_expectation}
\E[f(X)Y] \qq{\geq} c \E[f(X)].
\end{equation}
\end{lemma}
\begin{proof}
Using the tower property $\E[X] = \E[\E[X | Y]]$ combined with assumption \eqref{eq:onestep_expectation_c} we get
\[\E[f(X)Y] \qq{=} \E[\E[f(X)Y ~|~ X]] \qq{=} \E[f(X)\E[Y ~|~ X]] \qq{\stackrel{\eqref{eq:onestep_expectation_c}}{\geq}} \E[f(X) c] \qq{=} c\E[f(X)],\]
which concludes the proof.
\end{proof}

\newpage
\section{Notation Glossary}

\begin{table}[h]
\centering
\begin{tabular}{|c|c|c}
\hline
Notation & Description \\
\hline
$\R$ & the set of real numbers \\
$\R_+$ & the set of positive real numbers \\
$\extR$ & the set $\R \cup \{+\infty\}$ \\
\hline
$\bx$ & a vector \\
$x_i$ & the $i$-th entry of the vector $\bx$ \\
\hline
$\bX$ & a matrix \\
$\bX_{i:}$ & the $i$-th row of the matrix $\bX$ \\
$\bX_{:j}$ & the $j$-th column of the matrix $\bX$ \\
$X_{ij}$ & the entry at the $i$-th row and $j$-th column of the matrix $\bX$ \\
\hline
$[n]$ & a shorthand for the set $\{1, \dots, n\}$, usually containing the coordinates of the space \\
\hline
$\bx_{S}$ & a $|S|$-dimensional vector containing entries of $\bx$ with indices in $S$ \\
$\bx_{[S]}$ & the vector $\bx$ with the entries with indices outside $S$ zeroed out \\
\hline
$\bX_{S}$ & a $|S| \times |S|$ submatrix of $\bX$ containing only rows and columns with indices in $S$ \\
$\bX_{[S]}$ & the matrix $\bX$ with all entries on columns or rows outside of $S$ zeroed out \\
$\bX_{[S]}^{-1}$ & the $n \times n$ matrix containing $(\bX_S)^{-1}$ at the rows and columns indicated by $S$. \\
\hline
$f$ & a smooth function  from $\R^n$ to $\R$ (see Def.~\ref{ass:Msmoothness})\\
$g$ & a separable \eqref{ass:separability} and possibly non-smooth function from $\R^n$ to $\R$ \\
$F$ & the objective function from $\R^n$ to $\R$, defined as $f + g$ \eqref{eq:prox_problem} \\
\hline
$\nabla_S f(\bx)$ & a shorthand for $(\nabla f (\bx))_S$ \\
$\nabla_{[S]} f(\bx)$ & a shorthand for $(\nabla f (\bx))_{[S]}$ \\
\hline
$\xi(\bx)$ & the optimality gap defined in $\eqref{eq:optimality_gap}$ \\
$\lambda(\bx)$ & the auxilary function defined in \eqref{eq:lambda_ih8d} \\
$\mu(\bx)$ & the forcing function defined in \eqref{eq:GPL_function} \\
$\theta(S, \bx)$ & the proportion function defined in \eqref{eq:proportion_function}\\ 
\hline
$\cX^*$ & the set of global minimizers of $F$ \\
 $\cX$ & the set of vectors with nonzero $\lambda(\bx)$ \eqref{eq:set_Xalpha_proportion} \\
 \hline
$\lambda_{\min}(\bM)$ & the smallest eigenvalue of a square matrix $\bM$ \\
$\lambda_{\max}(\bM)$ & the largest eigenvalue of a square matrix $\bM$ \\
\hline
$\bM, L \bI$ & the matrices defining $\bM$-smoothness of $f$ \eqref{eq:Msmoothness} for smooth and non-smooth cases  \\
$L_\tau$  & the smoothness parameter defined as $\max_{S : |S| = \tau} \{\lambda_{\max}(\bM_S)\}$ \\ 
$\lambda_F, \lambda_f$ & strong convexity parameters of $F$ and $f$, respectively \\
\hline
$U_S(\bx, \bu)$ & the quadratic upper bound function at a given point $\bx$, defined in \eqref{def:US} \\
\hline
\end{tabular}
\caption{Notation Glossary.}
\label{tab:notation}
\end{table}

\end{document}